\documentclass{article}
\usepackage{arxiv}

\usepackage{mathtools}

\usepackage[utf8]{inputenc} % allow utf-8 input
\usepackage[T1]{fontenc}    % use 8-bit T1 fonts
\usepackage{hyperref}       % hyperlinks
\usepackage{url}            % simple URL typesetting
\usepackage{booktabs}       % professional-quality tables
\usepackage{nicefrac}       % compact symbols for 1/2, etc.
\usepackage[english]{babel}
\usepackage{ae}
\usepackage{eucal}
\usepackage[dvips]{graphicx}
\usepackage{epsfig}
\usepackage{graphicx}
\usepackage{amssymb}
\usepackage{amsthm}
\usepackage{amsfonts,amsmath}
\usepackage{amssymb}
\usepackage{a4}
\usepackage{xcolor}
\usepackage{amsmath}

\usepackage{multicol}
\usepackage{MnSymbol}

\usepackage{pifont}

\usepackage{dcolumn,amsthm}

\renewenvironment{proof}[1][Proof]{\noindent\textit{#1. } }{\hfill$\square$}

\newtheoremstyle{theorem}{6pt}{6pt}{\rm}{}{\sffamily}{ }{ }{}
\theoremstyle{theorem}
\newtheorem{theorem}{\sc Theorem}[section]

\newtheoremstyle{lemma}{6pt}{6pt}{\rm}{}{\sffamily}{ }{ }{}
\theoremstyle{lemma}
\newtheorem{lemma}{\sc Lemma}[section]

\newtheoremstyle{example}{6pt}{6pt}{\rm}{}{\sffamily}{ }{ }{}
\theoremstyle{example}

\newtheoremstyle{corollary}{6pt}{6pt}{\rm}{}{\sffamily}{ }{ }{}
\theoremstyle{corollary}
\newtheorem{corollary}{\sc Corollary}[section]

\newtheoremstyle{definition}{6pt}{6pt}{\rm}{}{\sffamily}{ }{ }{}
\theoremstyle{definition}
\newtheorem{definition}[theorem]{\sc Definition}

\newtheoremstyle{remark}{6pt}{6pt}{\rm}{}{\sffamily}{ }{ }{}
\theoremstyle{remark}
\newtheorem{remark}{\sc Remark}[section]

\newtheoremstyle{approximation}{6pt}{6pt}{\rm}{}{\sffamily}{ }{ }{}
\theoremstyle{approximation}

\newtheoremstyle{scheme}{6pt}{6pt}{\rm}{}{\sffamily}{ }{ }{}
\theoremstyle{scheme}

\usepackage{biblatex}

\usepackage{csquotes}

\title{Analysis of kinematics of mechanisms containing revolute joints}

\author{
   Jukka Tuomela \\
   Department of Physics and Mathematics\\
   University of Eastern Finland\\
   P.O. box 111, FI-80101 Joensuu, Finland  \\
   \text{jukka.tuomela@uef.fi} \\
}

\begin{document}

\maketitle

\begin{abstract}
Kinematics of rigid bodies can be analyzed in many different ways. The advantage of using Euler parameters is that the resulting equations are polynomials and hence computational algebra, in particular Gr\"obner bases, can be used to study them. The disadvantage of the Gr\"obner basis methods is that the computational complexity grows quite fast  in the worst case in the number of variables and the degree of polynomials. In the present article we show how to simplify computations when the mechanism contains revolute joints. The idea is based on the fact that the ideal representing the constraints of the revolute joint is not prime. Choosing the appropriate prime component reduces significantly the computational cost. We illustrate the method by applying it to the well known Bennett's and Bricard's mechanisms, but it can be applied to any mechanism which has revolute joints. 
\end{abstract}

\textbf{\emph Mathematics Subject Classification (2020)} 70B15, 13P10, 68W30

\keywords{Kinematical analysis, Mechanisms, Computational algebraic geometry}

%\thanks{The first author was supported by Finnish Cultural Foundation.}

\section{Introduction}

In the analysis of kinematics of mechanisms there are many ways to parametrize the rigid body. Usually the state of the rigid body is specified by the center of mass (or some other convenient point) and the orientation of the rigid body. However, it is also possible to treat the whole state at the same time by using dual quaternions and projective spaces  \cite{hegedus}. Here we consider just the orientations since in the analysis of constraints the main difficulty is how to deal with orientations. 

The orientation is given by an element of $\mathbb{SO}(3)$ and the problem is then to parametrize $\mathbb{SO}(3)$. It is well known that one coordinate system is not enough, in other words $\mathbb{SO}(3)$ is not homeomorphic to an open subset of $\mathbb{R}^3$. However, $\mathbb{SO}(3)$ is homeomorphic to $\mathbb{RP}^3$ and since $S^3$ is a double cover of $\mathbb{RP}^3$ this leads to a convenient parametrization using Euler parameters. This is the parametrization used in the present article. This has the advantage that the resulting equations representing the constraints are polynomials, and hence the tools of computational algebra, in particular Gr\"obner bases \cite{cox-li-os,singularbook},  are available to study the mechanisms. One can also represent $\mathbb{SO}(3)$ by Euler angles, Denavit-Hartenberg parameters \cite{denhar}, quaternions or complex $2\times 2$ matrices, since  $\mathbb{SU}(2)$ is homeomorphic to $S^3$. In fact some of the computations below could be interpreted in terms of quaternions, but because this would not be helpful in actual computations we will use exclusively Euler parameters.

The drawback with Gr\"obner basis methods is that their computational complexity is very bad in the worst case, see the discussion in \cite{rolspe} and references therein. On the other hand in practice many difficult problems can actually be solved so that the observed complexity is often very much smaller than the worst case. In spite of that the fact remains that also in practice the computing time grows rather quickly as a function of number of variables and the degree of polynomials. Our method improves the situation in both ways: if the mechanism contains revolute joints one can always reduce the degree of some polynomials in the system, and if a revolute joint is linked to the fixed rigid body one can always eliminate (at least) two variables. The idea is based on the fact that the constraints of the revolute joint define an ideal which is not prime, but has two prime components. Choosing systematically the correct prime component (and  there is only one component which corresponds to a given concrete physical situation) one can simplify the system considerably and in this way one can routinely analyze much bigger mechanisms with Gr\"obner basis methods than previously was possible.

In section 2 we introduce Euler parameters and recall some of their properties. We also review some relevant facts about ideals and varieties which are used in the sequel.  The basic tool here is the Gr\"obner basis of the given ideal and we have used {\sc Singular} \cite{singu-ohj} in all   polynomial computations.
In
section 3 we show how to decompose the ideal of the revolute joint in
general coordinates, generalizing the special case computed in \cite{mech-sci}.
In section 4 we apply this method to the analysis of the well-known Bricard's mechanism \cite{bricard}. We analyzed this mechanism the hard way in \cite{bric-nody}, but using the results of section 3 everything becomes easy. Then in section 5 we introduce Bennett's mechanism, originally introduced in \cite{bennett}; see for example   \cite{brhusch,percar,valayil,zhietal} and references therein for recent work on this topic. In section 6 we then analyze completely a certain subclass of Bennett's mechanism and our method yields an explicit characterization and parametrization of this subclass. As far as we know this is a new result. 
Some
conclusions and the significance of the decomposition for dynamics
simulation are discussed in section 7.

\section{Mathematical preliminaries}

\subsection{Notation and Euler parameters}
The standard unit vectors in $\mathbb{R}^{3}$ are denoted by $\mathbf{e}^{j}$.  If $y$ and $v\in\mathbb{R}^3$, then  $y\times v $ is the cross product and $\llangle y,v\rrangle$ is the inner product. If  $a\in\mathbb{R}^4$  then we  write $a=(a_0,\tilde a)$ where $\tilde a\in\mathbb{R}^3$, $\bar a=(a_0,-\tilde a)$ and if $v\in\mathbb{R}^3$ then $\hat v=(0,v)\in\mathbb{R}^4$. Let us define the following matrices.
\begin{align*}
\widetilde{H}(a)& =
\begin{pmatrix}
-a_{1} & a_{0} & -a_{3} & a_{2} \\
-a_{2} & a_{3} & a_{0} & -a_{1} \\
-a_{3} & -a_{2} & a_{1} & a_{0}
\end{pmatrix}
\quad ,\quad H(a)=
\begin{pmatrix}
-a_{1} & a_{0} & a_{3} & -a_{2} \\
-a_{2} & -a_{3} & a_{0} & a_{1} \\
-a_{3} & a_{2} & -a_{1} & a_{0}
\end{pmatrix}
\\[3mm]
R(a)=\widetilde{H}(a)H(a)^T & =
\begin{pmatrix}
a_{0}^{2}+a_{1}^{2}- a_2^2-a_3^2 & 2a_{1}a_{2}-2a_{0}a_{3} &
2a_{1}a_{3}+2a_{0}a_{2} \\
2a_{1}a_{2}+2a_{0}a_{3} & a_{0}^{2}-a_{1}^{2}+a_{2}^{2}-a_{3}^{2} &
2a_{2}a_{3}-2a_{0}a_{1} \\
2a_{1}a_{3}-2a_{0}a_{2} & 2a_{2}a_{3}+2a_{0}a_{1} & a_{0}^{2}-a_{1}^{2}-
a_2^2+a_{3}^{2}
\end{pmatrix}
\end{align*}
Now $R(a)\in \mathbb{SO}(3)$, if $|a|=1$.  Note that since the elements of $R(a)$ are homogeneous second order polynomials the action of $R(a)$ when $|a|\ne 1$ can be interpreted as a rotation followed by scaling.
We will also need the following matrices.
\begin{align*}
\tilde K(a) &=
\begin{pmatrix}
a_0 & a_1 & a_2 & a_3 \\
-a_{1} & a_{0} & -a_{3} & a_{2} \\
-a_{2} & a_{3} & a_{0} & -a_{1} \\
-a_{3} & -a_{2} & a_{1} & a_{0}
\end{pmatrix}
\quad ,\quad K(a)=
\begin{pmatrix}
a_0 & a_1 & a_2 & a_3 \\
-a_{1} & a_{0} & a_{3} & -a_{2} \\
-a_{2} & -a_{3} & a_{0} & a_{1} \\
-a_{3} & a_{2} & -a_{1} & a_{0}
\end{pmatrix}
\ .
\end{align*}
If $|a|=1$, then $\tilde K(a)$, $K(a)\in\mathbb{SO}(4)$.  One can check easily the following properties:
 \begin{equation}
\begin{aligned}
&  R(a)^T=R(\bar a)\\
& \tilde K(a)^Tb=K(b)^T a   \\
& R(c)=R(b)R(a)\quad \Longleftrightarrow \quad c=\pm K(b)^Ta \\
&\widetilde{K}(a)K(a)^T  =
\begin{pmatrix}
   |a|^2&0\\
   0&R(a)
\end{pmatrix}
\end{aligned}
\label{rot-kerto}
\end{equation}
It is convenient to establish the following facts which will be used below.
\begin{lemma} Let $y$, $v\in\mathbb{R}^3$ be nonzero and let $L= \tilde K(\hat v)K(\hat y)^T$. Then
\begin{enumerate}
\item $L$ is symmetric and $L^2=|v|^2|y|^2I$.
\item $L$ has two double eigenvalues $\lambda_\pm=\pm |v||y|$ and the eigenspaces of these eigenvalues  are two dimensional. 
\item Let $u$, $z\in \mathbb{R}^3$, $d=y\times z$, $q=u\times v$ and $L_0= \tilde K(\hat
u)^TK(\hat z)$; then
\[
  LL_0-L_0L=2\llangle u,v\rrangle K(\hat d)+2\llangle  y,z\rrangle\tilde K(\hat q)
\]
\item Let $La=|v||y|a$ where $v$ and $y$ are linearly independent;  then
\begin{align*}
& \big(|v|\,|y|- \llangle v,y \rrangle\big)a_0+ 
\llangle\tilde a, y\times v\rrangle=0 \\
& \big\llangle|y|\,v-|v|\,y, \tilde a \big\rrangle=0
\end{align*}
\end{enumerate}
\label{tekn-lem}
\end{lemma}

\begin{proof}
First three statements are easy verifications.  The conditions in the final statement are obtained by putting the matrix $
L-|v|\,|y|\,I$ to the row echelon form.
\end{proof}

\subsection{Ideals and varieties}

Here we quickly review the necessary tools that will be
needed. For more details we refer to \cite{cox-li-os,singularbook}.
Let us consider polynomials of variables $x_1,\dots,x_n$ with coefficients
in the field $\mathbb{K}$ which we suppose always to have characteristic zero and let us denote the ring of all such polynomials
by $\mathbb{A}=\mathbb{K}[x_1,\dots,x_n]$. The given polynomials $%
f_{1},\ldots ,f_{k}\in \mathbb{A}$ generate an \emph{ideal}:
\begin{equation*}
\mathcal{I}=\langle f_{1},\ldots ,f_{k}\rangle =\{f\in \mathbb{A}
~|~f=h_{1}f_{1}+\ldots +h_{k}f_{k},~\text{where}~h_{i}\in \mathbb{A}\}.
\end{equation*}
We say that the polynomials $f_{i}$ are \emph{generators} of $\mathcal{I}$
and as a set they are the \emph{basis} of $\mathcal{I}$. The \emph{variety}
corresponding to an ideal is
\begin{equation*}
\mathsf{V}(\mathcal{I})=\big\{a\in \mathbb{L}^{n}~|~f(a)=0\quad, \forall
~f\in \mathcal{I}\big\}\subset \mathbb{L}^{n}.
\end{equation*}
where $\mathbb{L}$ is some extension field of $\mathbb{K}$. In many cases $\mathbb{K}=\mathbb{Q}$ and in applied problems typically $\mathbb{L}=\mathbb{R}$ while from the theoretical point of view the choice $\mathbb{L}=\mathbb{C}$ is more convenient.  
The \emph{radical} of $\mathcal{I}$ is
\begin{equation*}
\sqrt{\mathcal{I}}=\{f\in \mathbb{A}~|~\exists ~m\geq 1~\text{such that}
~f^{m}\in \mathcal{I}\}.
\end{equation*}
An ideal $\mathcal{I}$ is a radical ideal if $\mathcal{I}=\sqrt{\mathcal{I}}$. Note that $\mathsf{V}(\mathcal{I})=\mathsf{V}(\sqrt{\mathcal{I}})$. An ideal $\mathcal{I}$ is prime, if $ab\in\mathcal{I}$ implies that either $a\in\mathcal{I}$ or $b\in\mathcal{I}$.
Evidently a prime ideal is always a radical ideal. Given two ideals their
\emph{sum} is
\begin{equation*}
\mathcal{I}+\mathcal{J}=\{f+g\in \mathbb{A}\,|\,f\in \mathcal{I}\ ,\ g\in
\mathcal{J}\}.
\end{equation*}
The geometric meaning of the definition is that
\begin{equation*}
\mathsf{V}(\mathcal{I}+\mathcal{J})=\mathsf{V}(\mathcal{I})\cap \mathsf{V}(
\mathcal{J}).
\end{equation*}
The following facts are fundamental:

\begin{itemize}
\item[(i)] every ideal is \emph{finitely generated}, i.e. it has a basis
with a finite number of generators. Moreover any ideal has a Gr\"obner basis. 

\item[(ii)] every radical ideal can be decomposed into a finite number of
prime ideals:
\begin{equation*}
\sqrt{\mathcal{I}}=I_{1}\cap \dots \cap I_{s}
\end{equation*}
where each $I_{\ell }$ is prime. This gives the decomposition of the variety
to \emph{irreducible components}:
\begin{equation}
V=\mathsf{V}(\mathcal{I})=\mathsf{V}(\sqrt{\mathcal{I}})=\mathsf{V}(I_{1})\cup
\dots \cup \mathsf{V}(I_{s}) \ .  
\label{decomposition}
\end{equation}
\end{itemize}

The  dimension of $V$ is denoted by  $\dim(V)$ and let us recall that if we have the irreducible decomposition as above then
\[
  \dim(V)=\max_{1\le j\le s}\dim(\mathsf{V}(I_j))\ .
\]
Note that $\dim(V)$ can be  computed, if the Gröbner basis of $\mathcal{I}$ is available. The local dimension is then
\[
  \dim_p(V)=\Big\{\max_{1\le j\le s}\dim(\mathsf{V}(I_j))\,|\, p\in \mathsf{V}(I_j)\Big\}\ .
\]
Let $\mathcal{I}\subset\mathbb{A}$. Then its $k$'th elimination ideal is
\begin{equation*}
\mathcal{I}_k=\mathcal{I}\cap \mathbb{K}[x_{k+1},\dots,x_n].
\end{equation*}
Geometrically this is related to projections. Let us set
\begin{equation*}
\pi_k\,:\,\mathbb{K}^n\to \mathbb{K}^{n-k}\quad,\quad (x_1,\dots,x_n)\to
(x_{k+1},\dots,x_n).
\end{equation*}
Now restricting the projection map to $ \mathsf{V}(\mathcal{I}) $ we always have $\pi_k\big(\mathsf{V}(\mathcal{I})\big)\subset \mathsf{V}(\mathcal{I}_k)$ so that the map $\pi_k\, :\, \mathsf{V}(\mathcal{I})\to \mathsf{V}(\mathcal{I}_k)$ is always well defined. Moreover we have
\begin{equation}
\overline{\pi_k\big(\mathsf{V}(\mathcal{I})\big)}=\mathsf{V}(\mathcal{I}_k).
\label{projektio}
\end{equation}
where the overline denotes the \emph{Zariski closure}. 
A basis of an elimination ideal can be computed using a suitable product
order. 
For example
\begin{equation*}
\mathcal{I}\subset \mathbb{K}%
[(x_{1},x_{2},x_{3})_{\mathsf{lex}},(y_{1},y_{2},y_{3},y_{4})_{\mathsf{degrev}}]
\end{equation*}%
means that this is a product order where we use lex order for variables $x_{i}$ and degree
reverse lex order for variables $y_{i}$.  If $G$ is the Gr\"{o}bner
basis of $\mathcal{I}$ with respect to this order then $G_{e}=G\cap
\mathbb{K}[y]$ is a Gr\"{o}bner basis for $\mathcal{I}_{3}=\mathcal{I}\cap
\mathbb{K}[y]$.

Let   $V$ be a variety and   let 
 $T_{p}V$ be the \emph{Zariski tangent space} of $V$
at $p$. 
\begin{definition}
A point $p\in V$ is \emph{singular} if $\dim (T_{p}V)>\dim _{p}(V)$.
Otherwise the point $p$ is \emph{regular}. The set of singular points of $V$
is denoted by $\Sigma (V)$. $V$  is regular (or smooth), if it has no singular points. 
\label{singu2}
\end{definition}

Recall that $\Sigma(V)$ is itself a variety and $\dim(\Sigma(V))<\dim(V)$. Hence "almost all" points of a variety are regular. 

Let $V=V_{1}\cup \dots \cup V_{\ell }$ be the decomposition into irreducible
components. Then there are basically two ways of a point $x$ of a general
variety $V$ to be singular: either $x$ is a singularity of an irreducible
component $V_{i}$ or it is an intersection point of irreducible components.
That is, the variety of singular points is
\begin{equation}
\Sigma (V)=\bigcup_{i=1}^{\ell }\Sigma (V_{i})\cup \bigcup_{i\neq
j}V_{i}\cap V_{j}.  
\label{singu}
\end{equation}

Once we have the irreducible decomposition it is easy to compute the
intersections (the second term in (\ref{singu})). In order to compute the
singular points of irreducible components (the first term in (\ref{singu}))
one needs the concept of \emph{Fitting ideals}.

Let $M$ be a matrix of size $k\times n$ with entries in $\mathbb{A}$.
The $m$th Fitting ideal of $M$, $\mathsf{F}_{m}(M)$, is the ideal generated
by the $m\times m$ minors of $M$. If now  $f_j\in\mathbb{A}$ we can then define a map $f=(f_{1},\ldots ,f_{k})\,:\,
\mathbb{K}^{n}\mapsto \mathbb{K}^{k}$ and its  differential (or Jacobian) is denoted by $df\in\mathbb{A}^{k\times n}$. 
The following result is usually called the \emph{Jacobian
criterion}.

\begin{theorem}  Let us suppose that $\mathcal{I}=\langle f_{1},\dots ,f_{k}\rangle $ is prime and let 
$V=\mathsf{V}(
\mathcal{I})$ be the corresponding irreducible variety.  If  $\mathsf{dim}(V)=n-m $, then 
 the singular variety of  $V$ is
\begin{equation*}
\Sigma (V)=\mathsf{V}\big(\mathcal{I}+\mathsf{F}_{m}(df)\big)=\mathsf{V}\big(
\mathcal{I}\big)\,\cap \,\mathsf{V}\big(\mathsf{F}_{m}(df)\big).
\end{equation*}
In particular if $\mathcal{I}+\mathsf{F}_{m}(df)=\mathbb{A}$ then $V$ is regular. \label{jac-crit}
\end{theorem}
Finally we need one concept which is not actually so  important in kinematic analysis, but we will need it in the final remarks at the end of the article.  Let 
\[
 \mathcal{I}=\langle f_{1},\ldots ,f_{m}\rangle \subset 
 \mathbb{K}[x_1,\dots,x_n]
 \quad\mathrm{and}\quad
V= \mathsf{V}(\mathcal{I})\ .
\]
The codimension of $V$ is $\mathsf{codim}(V)=n-\dim(V)$.
We say that $V$ is a complete intersection, if $\mathsf{codim}(V)=m$.  The concept of complete intersection corresponds to the intuitive idea that when one adds more equations the dimension of the solution set gets smaller. However, not all varieties are complete intersections, see the discussion in \cite{harris}.

\subsection{Goal of computations}

Let us then outline the overall strategy in the computations that follow. Since the computing time depends heavily on the number of variables it is always a good idea to eliminate as many variables as possible. Let 
\[
 \mathcal{I}=\langle f_{1},\ldots ,f_{m}\rangle \subset 
 \mathbb{K}[x_1,\dots,x_n]
\]
be a given ideal and $V( \mathcal{I}) $ the corresponding variety.  Let $ \mathsf{V}(\mathcal{I}_k)$ be the variety corresponding to the $k$th elimination ideal and let  $\pi_k\,:\, \mathsf{V}(\mathcal{I}) \to
   \mathsf{V}(\mathcal{I}_k)$ be the projection map.  We will call variables $x_{k+1},\dots,x_n$ essential, if the following holds:
\begin{itemize}
\item[($\mathsf{ess}_1$)] there is a polynomial map
\begin{equation}
   h\ :\ \mathsf{V}(\mathcal{I}_k) \to
   \mathsf{V}(\mathcal{I})
\label{sektio}
\end{equation}
such that  $\pi_k\circ h=\mathsf{id}$ and  $h\circ \pi_k=\mathsf{id}$.
\item[($\mathsf{ess}_2$)] $k$ is as big as possible.
\end{itemize}
 We could then say that the variety $\mathsf{V}(\mathcal{I}_k)$ and the map $h$ contain all the essential information about $\mathsf{V}(\mathcal{I})$. For example in this situation $\mathsf{V}(\mathcal{I}_k)$ is regular or complete intersection if and only if $\mathsf{V}(\mathcal{I})$ is regular or complete intersection.

 \begin{remark}
 Of course the choice of the essential variables is not unique, and in concrete situations it may not be immediately obvious which variables should be eliminated. Moreover the resulting $k$ may not be equal for all choices of variables. Let us give an example of such situation. Let
 \begin{align*}
 &   \mathcal{I}=\langle y-x^2\rangle\subset \mathbb{Q}[x,y]
  \quad,\quad
    \mathsf{V}(\mathcal{I})\subset \mathbb{R}^2
 \\
& \pi_x\,:\,\mathbb{R}^2\to \mathbb{R}\ ,\ \pi_x(x,y)=x
 \quad,\quad
  \pi_y\,:\,\mathbb{R}^2\to \mathbb{R}\ ,\ \pi_y(x,y)=y
 \end{align*}
 Now we can choose $x$ as an essential variable: $\mathcal{I}_1=\{0\}$, $\mathsf{V}(\mathcal{I}_1)=\mathbb{R}$ and $h=(x,x^2)$. However, $y$ cannot be chosen as an essential variable. This fails in two ways: here also $\mathcal{I}_1=\{0\}$ and $\mathsf{V}(\mathcal{I}_1)=\mathbb{R}$, but $\pi_y\,:\,\mathsf{V}(\mathcal{I})\to\mathbb{R}$ is not injective and the image is not Zariski closed.
\hfill \ding{71}
\end{remark}

Another goal is to compute the prime decomposition of the given ideal. Now mechanisms are designed to operate in a certain way, and perhaps it is a bit counterintuitive that in spite of that typically the relevant ideals are not prime. However, often only one of the prime components is physically relevant; these nonphysical cases could be called spurious components. 

Since the direct computation of the prime decomposition is usually not possible because of the very bad time complexity of the algorithm we proceed as follows.  Let again $\mathcal{I}$ be the given ideal and let $J\subset \mathcal{I}$. Suppose that we can compute the decomposition $J=J_0\cap J_1$; then we have
\begin{equation}
\mathsf{V}(\mathcal{I})=
\mathsf{V}(\mathcal{I}+J)=
\mathsf{V}(\mathcal{I}+J_0)\cup
\mathsf{V}(\mathcal{I}+J_1)\ .
\label{dekomp}
\end{equation}
Then we simply check if $\mathsf{V}(\mathcal{I}+J_0)$ or $\mathsf{V}(\mathcal{I}+J_1)$ is the relevant component and proceed the analysis with it. Even in the case that both varieties are relevant we have in any case succeeded in splitting the problem to two subsystems which are easier to analyze than the original one. 

Intuitively one should choose $J$ such that it contains only polynomials which depend on few variables $x_j$. Note that $J$ does not have to be an elimination ideal. Fortunately in mechanical systems the constraints are such that usually the polynomials indeed depend on just a small subset of all variables. While the choice of an appropriate $J$ might still not be easy in general, in the present paper we show how to find such $J$ when the system contains revolute joints and show how to compute $J_0$ and $J_1$.

\section{Decomposition of the revolute joint in arbitrary local frames}

Let $\mathcal{B}_0$ and $\mathcal{B}_1$ be two rigid bodies connected by a revolute joint $J$.  Let $\chi$ be the vector giving the axis of rotation of $J$  and let $\mathsf{span}\{\eta,\xi\}$ be the plane orthogonal to $\chi$. Further let $a$ (resp. $b$) be the Euler  parameters of $\mathcal{B}_0$ (resp. $\mathcal{B}_1$). 
To each body we associate a local coordinate system $\mathcal{C}_\ell$ and we express $\chi$ in $\mathcal{C}_1$, and $\eta$ and $\xi$ in $\mathcal{C}_0$. In this situation the constraints of the revolute joint are given by the ideal 
\begin{equation}
\begin{aligned}
I_{ab}& =\langle q_{1},q_{2},|a|^{2}-1,|b|^{2}-1\rangle \quad \mathrm{where}
\\
q_{1}& =\llangle R(a)\eta,R(b)\chi \rrangle \\
q_{2}& =\llangle R(a)\xi,R(b)\chi \rrangle
\end{aligned}
\label{alkup}
\end{equation}
In \cite{mech-sci} we discovered that this ideal is not prime when $\eta=\mathbf{e}^1$, $\xi=\mathbf{e}^2$ and $\chi=\mathbf{e}^3$ and  computed its prime decomposition. Let us denote the Euler parameters describing this situation by $u$ and $z$  and consider the ideal
\begin{align*}
I_{\mathsf{rev}}&=\langle p_1,p_2,|u|^2-1,|z|^2-1\rangle\quad\mathrm{where}
\\
p_1&=\llangle R(u)\mathbf{e^1}, R(z)\mathbf{e^3}\rrangle \\
p_2&=\llangle R(u)\mathbf{e^2}, R(z)\mathbf{e^3}\rrangle
\end{align*}
The decomposition has the following form:
\begin{equation}
\begin{aligned} I_{\mathsf{rev}} &= I_{\mathsf{rev}}^1\cap
I_{\mathsf{rev}}^2 \\ I_{\mathsf{rev}}^j &=\langle
k^j_1,k^j_2,k^j_3,k^j_4,k^j_5,|u|^2-1,|z|^2-1\rangle \end{aligned}
\label{hajo}
\end{equation}
where the polynomials $k^j_m$ are as follows
\begin{align*}
&k^1_1=z_1u_3+z_0u_2-z_3u_1-z_2u_0 & k^2_1&=z_0u_3-z_1u_2+z_2u_1-z_3u_0 \\
&k^1_2=z_2u_3-z_3u_2-z_0u_1+z_1u_0 & k^2_2&=z_3u_3+z_2u_2+z_1u_1+z_0u_0 \\
&k^1_3=z_3z_1+z_2z_0-u_3u_1-u_2u_0 & k^2_3&=z_3z_1+z_2z_0+u_3u_1+u_2u_0 \\
&k^1_4=z_3z_2-z_1z_0-u_3u_2+u_1u_0 & k^2_4&=z_3z_2-z_1z_0+u_3u_2-u_1u_0 \\
&k^1_5=z_2^2+z_1^2-u_2^2-u_1^2 & k^2_5&=z_3^2+z_0^2-u_2^2-u_1^2
\end{align*}
Note that  $
\mathsf{V}(I_{\mathsf{rev}}^1)\cap \mathsf{V}(I_{\mathsf{rev}
}^2)=\emptyset $. This is in fact quite important from practical point of view because it shows that in any given situation we need only one of the components: if the system is in one component initially then it cannot move to the other component at later time. 

From the physical point of view the existence of two components is perfectly natural: it corresponds to choice of orientation of the axis of rotation. If one fixes local coordinates in two bodies then one can assemble them in two ways: so that the orientations of the axis rotation is the same in both bodies or such that they are opposite. After the assembly the choice of orientation cannot be changed which is the physical reason for the property $
\mathsf{V}(I_{\mathsf{rev}}^1)\cap \mathsf{V}(I_{\mathsf{rev}
}^2)=\emptyset $.

\begin{remark} The degree of polynomials $k_j^\ell$ is only two which is somewhat surprising. The degree of $p_j$ in $I_{\mathsf{rev}}$ is four, but in general the generators in the prime decomposition are not of lower degree, or otherwise simpler than the original generators. 

Passing from degree four generators to degree two generators may sound like a modest improvement, but actually this reduces significantly the computational complexity.
\hfill \ding{71}
\end{remark}

Now geometrically one may view that $I_{\mathsf{rev}}$ and $I_{ab}$ describe really the same thing, using different coordinate systems. Hence $I_{ab}$ cannot be prime and our goal is to compute its prime components. To do this we  have to find appropriate coordinate transformations so that one can express the decomposition of $I_{ab}$ in terms of decomposition of $I_{\mathsf{rev}}$.
Note that in practice the direct computation of the prime decomposition is not feasible in the general case, at least with standard computers, because of the very bad time complexity of the prime decomposition algorithm.

To find the appropriate coordinate transformations we have to solve the following problems:
\begin{enumerate}
\item  given $\chi$ find $\beta$  and $\kappa$ such that $\kappa\chi=R(\beta)\mathbf{e}^3$.
\item given orthogonal $\eta$ and $\xi$ find $\alpha$, $\kappa_1$ and $\kappa_2$ such that $\kappa_1\eta=R(\alpha)\mathbf{e}^1$ and  $\kappa_2\xi=R(\alpha)\mathbf{e}^2$.
\end{enumerate}
Note that we have to assume that $\eta$ and $\xi$ are orthogonal because $\mathbf{e}^1$ and $\mathbf{e}^2$ are orthogonal and the rotation preserves angles. However, this is no real restriction since we are given some two dimensional subspace and we can always choose its basis $\{\eta,\xi\}$ such that it is orthogonal.

Let us solve both problems slightly more generally.
\begin{lemma}
Let $y$, $v\in\mathbb{R}^3$   and $L\beta=|v||y|\beta $ where $L$ is as in Lemma \ref{tekn-lem}. Then
\[
     \frac{|\beta|^2|v|}{|y|}\,y=R(\beta)v
\]
\label{1dim-lem}
\end{lemma}
\begin{proof}
The condition $\kappa\,y=R(\beta)v$ is equivalent to $\kappa\,\hat y=\tilde
K(\beta)K(\beta)^T\hat v$ by \eqref{rot-kerto}. The result then follows from the following computation:
\begin{align*}
\tilde  K(\beta) K(\beta)^T \hat v=&\tilde  K(\beta)\tilde K(\hat
v)^T \beta=
\frac{1}{|v||y|}\,\tilde  K(\beta)\tilde K(\hat
v)^T L\beta\\
=&
\frac{1}{|v||y|}\,\tilde  K(\beta)\tilde K(\hat
v)^T \tilde K(\hat v)K(\hat y)^T \beta=
\frac{|v|}{|y|}\,\tilde  K(\beta)K(\hat y)^T \beta\\
=& \frac{|v|}{|y|}\,\tilde  K(\beta)\tilde K(\beta)^T \hat y=
  \frac{|\beta|^2|v|}{|y|}\,\hat y 
\end{align*}
\end{proof}

\begin{remark} The reader who is familiar with quaternions will recognize that the above Lemma could be expressed with quaternions. Note that $K(a)b$ is essentially quaternion multiplication.
\hfill \ding{71}
\end{remark}

Consider the ideals \eqref{alkup} and \eqref{hajo}. We need to find $\beta$ such that $R(\beta)\mathbf{e}^{3}=\kappa\,\chi$ for some $\kappa$. Now in our application we are interested only in subspaces spanned by the vectors and not by vectors themselves. 
Hence without loss of generality we may assume that $|\chi|=1$. Of course if $\chi$ and $\mathbf{e}^{3}$ are linearly dependent then there is no need to do anything: one simply sets $\chi=\mathbf{e}^{3}$. In the interesting case we have
\begin{lemma} If $\chi$ and $ \mathbf{e}^{3}$ are linearly independent and $|\chi|=1$, then $|\beta|^2\chi=R(\beta) \mathbf{e}^{3}$, if we take 
\[
\beta =\big(1+\chi _{3},-\chi _{2},\chi _{1},0\big) \ .
\]
\label{beta}
\end{lemma}

\begin{proof}
Let $L= \tilde K(\mathbf{\hat{e}}^{3})
K(\hat \chi)^T$. It is easy to check that $L\beta=\beta$ where $\beta $ is as above and the result then follows from  Lemma \ref{1dim-lem}.
\end{proof}

Note that necessarily $\beta\ne 0$. Of course one could also normalize $\beta$ to be a unit vector, but it turns out to be convenient in the symbolic computations that follow to leave $\beta$ as above. 

Then we need to find $\alpha$ and $\kappa_j$ such that $R(\alpha)\mathbf{e}^1=\kappa_1\,\eta$ and $R(\alpha)\mathbf{e}^2=\kappa_2\,\xi$. Here also it is convenient to normalize that $|\eta|=|\xi|=1$. 
\begin{lemma}
Suppose that $\{\eta,\xi\}$ is an orthonormal basis. 
Then there is a nonzero $\alpha$ such that 
\begin{equation*}
|\alpha|^2\eta=R(\alpha)\mathbf{e}^1\quad \mathrm{and}\quad
|\alpha|^2 \xi=R(\alpha)\mathbf{e}^2\ .
\end{equation*}
Moreover $\alpha$ is essentially unique in the sense that any two such $\alpha$s are linearly dependent.
\label{taso-lem}
\end{lemma}

\begin{proof}
Let $L_1= \tilde K(\hat{\mathbf{e}}^1)^TK(\hat\eta)$ and $L_2= \tilde K(\hat{\mathbf{e}}^2)^TK(\hat\xi)$. If we can find a nonzero $\alpha$ such that $L_1\alpha=L_2\alpha=\alpha$  the result follows from  Lemma \ref{1dim-lem}. Let us set
\begin{align*}
  I_{\eta\xi}=&\langle |\eta|^2-1,|\xi|^2-1,\eta_1\xi_1+\eta_2\xi_2+\eta_3\xi_3\rangle
   \subset\mathbb{Q}[\eta,\xi]\\
M=&\begin{pmatrix}
     L_1-I\\
     L_2-I
\end{pmatrix}
\end{align*}
 Let $\mathcal{N}(M)$ be the nullspace of $M$. Let us compute the rank of $M$ when $(\eta,\xi)\in \mathsf{V}(I_{\eta\xi})$. First we compute $\mathsf{F}_4(M)$ and reduce the ideal with respect to $I_{\eta\xi}$; this gives the zero ideal and hence the rank of $M$ is at most three on $\mathsf{V}(I_{\eta\xi})$. But this implies that the dimension of $\mathcal{N}(M)$ is at least one  on $\mathsf{V}(I_{\eta\xi})$ and the required nonzero $\alpha$ exists.
 
 Then  we compute $\mathsf{F}_3(M)$ and reduce it which gives the whole ring $\mathbb{Q}[\eta,\xi]$ which implies that  the rank of $M$ is precisely three on $\mathsf{V}(I_{\eta\xi})$ and hence $\mathcal{N}(M)$  is one dimensional so that any two elements of the nullspace are linearly dependent. 
\end{proof}

 In the concrete computations it is convenient to have more explicit formulas for $\alpha$. 
If now $\{\eta,\xi\}$ spans the same subspace as $\{\mathbf{e}^1,\mathbf{e}^2\}$ then there is no need to do anything; one can simply rename $\eta=\mathbf{e}^1$ and $\xi=\mathbf{e}^2$.
 Otherwise, by renaming the variables if necessary, we may without loss of generality suppose that $\eta$ does not belong to the subspace spanned by $\{\mathbf{e}^1,\mathbf{e}^2\}$. In this case we obtain
 \begin{lemma} Suppose that $\{\eta,\xi\}$ is an orthonormal basis such that $\eta\not\in \mathsf{span}\{\mathbf{e}^1,\mathbf{e}^2\}$. Let
 \[
\alpha^1 = 
\begin{pmatrix}
\eta_2-\xi_1 \\[1mm]
\eta_2\xi_3-\xi_2\eta_3+\eta_3 \\[1mm]
\xi_1\eta_3-\xi_3\eta_1+\xi_3 \\[1mm]
1-\eta_2\xi_1+\xi_2\eta_1-\xi_2-\eta_1
\end{pmatrix}
\quad \mathrm{and}\quad
\alpha^2 =
\begin{pmatrix}
\eta_2\xi_3-\xi_2\eta_3-\eta_3 \\[1mm]
\eta_2+\xi_1 \\[1mm]
1+\eta_2\xi_1-\xi_2\eta_1+\xi_2-\eta_1\\[1mm]
\xi_1\eta_3-\xi_3\eta_1+\xi_3 
\end{pmatrix}\ .
\]
Then at least one of $\alpha^j$   is nonzero and gives the required $\alpha$ in Lemma \ref{taso-lem}.
\label{alfa}
\end{lemma}
\begin{proof}
Let $L_1= \tilde K(\mathbf{e}^1)^TK(\eta)$  and $L_2= \tilde K(\mathbf{e}^2)^TK(\xi)$.  Simple computations using Lemma \ref{tekn-lem} show that vectors 
\[
    w^1=\big(\eta_2,\eta_3,0,1-\eta_1\big)
    \quad \mathrm{and }\quad
    w^2=\big(-\eta_3,\eta_2,1-\eta_1,0\big)
\] 
span the eigenspace of $L_1$ corresponding to the eigenvalue one. Then we compute $\alpha^j=w^j+L_2w^j$ which gives the formulas above. By Lemma \ref{tekn-lem} $L_1L_2=L_2L_1$ and $L_1^2=L_2^2=I$ which implies that $L_1\alpha^j=L_2\alpha^j=\alpha^j$.  Now $\eta_3\ne0$ because we suppose that $\eta$ does not belong to the subspace spanned by $\{\mathbf{e}^1,\mathbf{e}^2\}$. This implies that  the second element of $\alpha^1$ and the first element of $\alpha^2$ cannot both be zero, and hence $\alpha^1$ and  $\alpha^2$ cannot both be zero. 
\end{proof}

\begin{remark} Of course typically both $\alpha^j$ in the above Lemma are nonzero, but in that case it does not matter which one is chosen. Also one may take some linear combination of them, if convenient. Note also that $\alpha^j$ are always linearly dependent on $\mathsf{V}(I_{\eta\xi})$ although this is not obvious by simply looking at the formulas. 

\hfill \ding{71}
\end{remark}

The following consequence of Lemma \ref{taso-lem} can perhaps be occasionally useful.
\begin{corollary} Let  $\{v^1,v^2\}$  and $\{y^1,y^2\}$ be orthogonal bases. Then there is a nonzero $\alpha$ such that 
\[
\frac{|\alpha|^2|v^1|}{|y^1|}\, y^1=R(\alpha)v^1\quad \mathrm{and}\quad
\frac{|\alpha|^2|v^2|}{|y^2|}\, y^2=R(\alpha)v^2\ .
\]
\end{corollary}
\begin{proof}
By Lemma \ref{taso-lem} we can find nonzero $a$ and $b$ such that
\[
   \frac{|a|^2}{|v^j|}\,v^j=R(a)\mathbf{e}^j
   \quad\mathrm{and} \quad
    \frac{|b|^2}{|y^j|}\,y^j=R(b)\mathbf{e}^j\ .
\]
But this implies that 
\[
 \frac{|b|^2|a|^2|v^j|}{|y^j|}\,y^j=R(b)R(a)^Tv^j=
     R(b)R(\bar a)v^j
\]
Hence we can take $\alpha=K(b)^T\bar a$. 
\end{proof}

Now choosing $\alpha$ and $\beta$ as in Lemmas \ref{beta} and \ref{alfa}  the polynomials  $q_1$ and $q_2$ in \eqref{alkup} can be written as
\begin{align*}
q_{1}& =\llangle R(a)R(\alpha )\mathbf{e^{1}},R(b)R(\beta )\mathbf{e^{3} }\rrangle \\
q_{2}& =\llangle R(a)R(\alpha )\mathbf{e^{2}},R(b)R(\beta )\mathbf{e^{3}}\rrangle\ .
\end{align*}
Next let us set
\begin{equation}
u=\frac{1}{|\alpha |}\,\widetilde{K}(\alpha )^{T}a=
\frac{1}{|\alpha |}\,K(a )^{T}\alpha
\quad \mathrm{and}\quad z=
\frac{1}{|\beta |}\,\widetilde{K}(\beta )^{T}b=
\frac{1}{|\beta |}\,K(b )^{T}\beta\ .
  \label{muunnos}
\end{equation}
Hence using the properties \eqref{rot-kerto} we obtain
\begin{equation*}
R(u)=R(a)R(\frac{\alpha }{|\alpha |})=\frac{1}{|\alpha |^{2}}\,R(a)R(\alpha
)\quad \mathrm{and}\quad R(z)=R(b)R(\frac{\beta }{|\beta |})=\frac{1}{|\beta
|^{2}}\,R(b)R(\beta )
\end{equation*}
and hence our original ideal \eqref{alkup} is of the form \eqref{hajo} in the new variables. Then finally using \eqref{muunnos} we substitute the values of $u$ and $z$ in the decomposition and thus obtain the required decomposition in original variables. Note that due to the nature of the
generators of the decomposition we do not need to evaluate the square root
implicit in $|\alpha |$ and $|\beta |$; it is sufficient to use $|\beta|^2=2(1+\chi_3) $, 
\[
 |\alpha^1|^2=4\big(1-\eta_1-\xi_2
    +\eta_1\xi_2-\eta_2\xi_1\big)
    \quad\mathrm{and}\quad
     |\alpha^2|^2=4\big(1-\eta_1+\xi_2
    -\eta_1\xi_2+\eta_2\xi_1\big)\ .
\]
Incidentally the above formulas also clearly imply that $\alpha^1$ and $\alpha^2$ cannot both be zero. To simplify the formulas $|\alpha^j|^2$ have been reduced with respect to $I_{\eta\xi}$.

\section{Bricard's mechanism}

Let us now consider the Bricard's mechanism. This is because the computations are quite simple so that one can clearly illustrate the ideas of the previous section. It is seen that the analysis of the Bricard's mechanism presented in   \cite{bric-nody}  becomes quite easy using the ideas of the previous section.

We can specify the mechanism as follows:
\begin{enumerate}
\item There are 6 rigid bodies $\mathcal{B}_\ell$, $0\le j\le 5$, connected by 6 revolute joints $J_\ell$ so that $\mathcal{B}_\ell$ is between $J_\ell$ and $J_{\ell+1}$ (with $J_6=J_0$).
\item The joints at the initial position are at points $p^\ell$ given by
\begin{align*}
  p^0=&(0,0,1)& p^1=&(1,0,1)
  & p^2=&(1,0,0)\\
  p^3=&(1,1,0)& p^4=&(0,1,0)
  & p^5=&(0,1,1)
\end{align*}
We also set $v^\ell=p^{\ell+1}-p^\ell$ (with $p^6=p^0$).
\item The axes of rotations at the initial position are given by
\begin{align*}
   \chi^0=&\mathbf{e}^3&
   \chi^1=&\mathbf{e}^2&
   \chi^2=&\mathbf{e}^1\\
      \chi^3=&\mathbf{e}^3&
      \chi^4=&\mathbf{e}^2&
      \chi^5=&\mathbf{e}^1
\end{align*}
Hence the planes orthogonal to them are given by
\begin{align*}
\{\eta^0,\xi^0\}=&\{\mathbf{e}^1,\mathbf{e}^2\}&
\{\eta^1,\xi^1\}=&\{\mathbf{e}^1,\mathbf{e}^3\}&
\{\eta^2,\xi^2\}=&\{\mathbf{e}^2,\mathbf{e}^3\}\\
\{\eta^3,\xi^3\}=&\{\mathbf{e}^1,\mathbf{e}^2\}&
\{\eta^4,\xi^4\}=&\{\mathbf{e}^1,\mathbf{e}^3\}&
\{\eta^5,\xi^5\}=&\{\mathbf{e}^2,\mathbf{e}^3\}
\end{align*}
\item We regard the body $\mathcal{B}_0$ as fixed and let $a$, $b$, $c$, $d$ and $e$ denote the Euler parameters of $\mathcal{B}_1,\dots,\mathcal{B}_5$. 
\item For each body $\mathcal{B}_\ell$, $1\le \ell\le 5$ we introduce the local frame $\{\mathbf{e}^{1,\ell},\mathbf{e}^{2,\ell},\mathbf{e}^{3,\ell}\}$ and we choose $\mathbf{e}^{1,\ell}$ to be parallel to $v^\ell$.
\item The frame $\{\eta^\ell,\xi^\ell,\chi^\ell\}$ in local coordinates of the body $k$ is denoted $\{\eta^{\ell,k},\xi^{\ell,k},\chi^{\ell,k}\}$. Note that these frames are orthonormal. 
\end{enumerate}
The first step is to compute the values of Euler parameters at the initial position. For $\mathcal{B}_1$ we should then solve $v^1=R(a_{\mathsf{init}})\mathbf{e}^{1,1}$ and similarly for other bodies. These equations are so simple that one readily sees that one can take
\begin{align*}
  a_{\mathsf{init}}=&\big(1/2,1/2,1/2,-1/2\big)&
   b_{\mathsf{init}}=&\big(1/2,1/2,1/2,1/2\big)&
   c_{\mathsf{init}}=&\big(0,0,0,1\big)\\
    d_{\mathsf{init}}=&\big(1/2,1/2,-1/2,1/2\big)&
     e_{\mathsf{init}}=&\big(1/2,1/2,-1/2,-1/2\big)
\end{align*}
We can now define
\[
  I_{\mathsf{init}}=\langle a-a_{\mathsf{init}},
  b-b_{\mathsf{init}},c-c_{\mathsf{init}},
  d-d_{\mathsf{init}},e-e_{\mathsf{init}}\rangle
  \subset \mathbb{Q}[a,b,c,d,e]
\]
which then gives us the point $\mathsf{V}( I_{\mathsf{init}})$. Moreover we can now express the local coordinates in terms of global coordinates; for example  we have
\[
   \eta^{2,1}=R(a_{\mathsf{init}})^T\eta^2=
   R(a_{\mathsf{init}})^T\mathbf{e}^2=
   (0,0,-1)\ .
\]
Let us then consider the joint $J_1$. The constraints are in this case
\begin{align*}
  \llangle \eta^1,R(a)\chi^{1,1}\rrangle=&
   \llangle \mathbf{e}^1,R(a)R(a_{\mathsf{init}})^T\chi^1\rrangle=
    -  \llangle \mathbf{e}^1,R(a)\mathbf{e}^3\rrangle=-2(a_0a_2+a_1a_3)\\
      \llangle \xi^1,R(a)\chi^{1,1}\rrangle=&
    -  \llangle \mathbf{e}^3,R(a)\mathbf{e}^3\rrangle=-a_0^2+a_1^2+a_2^2-a_3^2
\end{align*}
Hence we have an ideal
\[
  I_a=\langle |a|^2-1,a_0a_2+a_1a_3,
  -a_0^2+a_1^2+a_2^2-a_3^2\rangle\subset
  \mathbb{Q}[a_0,a_1,a_2,a_3]
\]
Of course this is a special case of the system  \eqref{alkup} where one of the rotation matrices is identity. However, here the ideal is so simple that we can directly compute its prime decomposition:\footnote{{\sc Singular} has a command \textsf{minAssGTZ} for this purpose.} 
\begin{align*}
  I_a=&I_{a,0}\cap I_{a,1}\subset
  \mathbb{Q}[a_0,a_1,a_2,a_3]\\
  I_{a,0}=&\langle 2a_2^2+2a_0^2-1,
  a_1+a_0,a_3-a_2\rangle&
    I_{a,1}=&\langle 2a_2^2+2a_0^2-1,
  a_1-a_0,a_3+a_2\rangle
\end{align*}
Now clearly $a_{\mathsf{init}}\in \mathsf{V}(I_{a,1})$ so we discard $I_{a,0}$ and continue our analysis with $I_{a,1}$, according to the idea in the formula \eqref{dekomp}. Recall that $
\mathsf{V}(I_{\mathsf{rev}}^1)\cap \mathsf{V}(I_{\mathsf{rev}
}^2)=\emptyset $ in \eqref{hajo} so that $a_{\mathsf{init}}$ can belong only to one of the components.

Let us then consider the joint $J_2$ where both variables $a$ and $b$ appear and we really need the method given in the previous section. The constraints are now 
\begin{align*}
  \llangle R(a)\eta^{2,1},R(b)\chi^{2,2}\rrangle   =& 
  -\llangle R(a)\mathbf{e}^3,R(b)\mathbf{e}^3\rrangle  =0\\
     \llangle R(a)\xi^{2,1},R(b)\chi^{2,2}\rrangle=&
       -\llangle R(a)\mathbf{e}^1,R(b)\mathbf{e}^3\rrangle =0
\end{align*}   
The resulting polynomials are already rather large so we do not write them explicitly; anyway we will not need them.    Now in this case we already have $\chi=\mathbf{e}^3$ so we do not need Lemma \ref{beta} and hence we can take $z=b$ in the formula \eqref{muunnos}. Then using Lemma \ref{alfa} with $\eta=(0,0,-1)$ we can take $w=(1,0,1,0)$ which gives $\alpha=(1,-1,1,1)$ and 
\[
  u=\frac{1}{2}\,\big(
  -a_3-a_2+a_1+a_0,
  -a_3+a_2+a_1-a_0,
  -a_3+a_2-a_1+a_0,
   a_3+a_2+a_1+a_0\big)
\]
Substituting this to \eqref{hajo}  we obtain
\begin{align*}
k^1_1=&(b_3+b_2+b_1-b_0)a_3+
(-b_3+b_2+b_1+b_0)a_2
+(-b_3-b_2+b_1-b_0)a_1
+(b_3-b_2+b_1+b_0)a_0\\
k_2^1=&(b_3+b_2-b_1+b_0)a_3
+(-b_3+b_2-b_1-b_0)a_2
+(b_3+b_2+b_1-b_0)a_1+
(-b_3+b_2+b_1+b_0)a_0\\
k_3^1=&  b_3b_1+b_2b_0-a_2a_1+a_3a_0  \\
k_4^1=&  2b_3b_2-2b_1b_0+a_3^2-a_2^2+a_1^2-a_0^2  \\
k_5^1=&   2b_2^2+2b_1^2-a_3^2+2a_3a_2-a_2^2-a_1^2+2a_1a_0-a_0^2
\end{align*}
Now we just check  that substituting $a_{\mathsf{init}}$ and $b_{\mathsf{init}}$ to these polynomials gives zero so that this is the
 correct component and we discard the other prime component. At this point one can put  together these prime components:
\[
  I_1=I_{a,1}+\langle k_1^1,\dots,k_5^1,|a|^2-1,|b|^2-1\rangle\subset \mathbb{Q}[a_3,a_1,b_3,b_1,b_2,b_0,a_2,a_0]
\]
and compute the Gröbner basis using the lex order; this gives
\begin{align*}
   I_1=&\langle g_0,\dots,g_9\rangle
   \quad\mathrm{where}\\
   g_0=&2a_2^2+2a_0^2-1&
g_1=&b_2a_0-b_0a_2\\
g_2=&b_1^2+b_2^2+2a_0^2-1&
g_3=&2b_3a_0^2-b_3+2b_1a_2a_0\\
g_4=&b_3a_2-b_1a_0&
g_5=&b_3b_2-b_1b_0\\
g_6=&b_3b_1+b_2b_0-2a_2a_0&
g_7=&b_3^2+b_0^2-2a_0^2\\
g_8=&a_1-a_0&
g_9=&a_3+a_2
\end{align*}
Note that this is already a substantial simplification compared to the original constraints. Continuing in this way with other joints, and adding the constraints due to the loop in the mechanism, we arrive at the final Gröbner basis:
\begin{align*}
I=&\langle h_0,\dots,h_{18}\rangle\subset\mathbb{Q}[a,b,c,d,e]\\
h_0=&2a_2^2+2a_0^2-1&
h_1=&16a_0^2(c_2^2-2c_2a_2
-a_0^2)+1\\
h_{2}=&a_1-a_0&
h_{3}=&a_3+a_2\\
h_4=&c_0-4a_0(c_2^2-2c_2a_2-a_0^2)-a_0&
h_5=&c_1+4a_0(c_2^2-2c_2a_2-a_0^2)+a_0\\
h_{6}=&c_3+c_2-2a_2\\
h_{7}=&2b_0-1&
h_8=&b_1+2a_2(c_2-a_2)\\
h_9=&b_2+8a_0a_2(c_2^2-
2c_2a_2-a_0^2)&
h_{10}=&b_3+2a_0(c_2-a_2)\\
h_{11}=&d_0+2a_2(c_2-a_2)&
h_{12}=&2d_1-1\\
h_{13}=&d_2+2a_0(a_2-c_2)&
h_{14}=&d_3+8a_0a_2(c_2^2-
2 c_2 a_2-a_0^2)
\\
h_{15}=&e_0+4a_0(c_2^2-2c_2a_2-a_0^2)&
h_{16}=&e_1+4a_0(c_2^2-2c_2a_2-a_0^2)&
\\
h_{17}=&e_2-c_2+a_2&
h_{18}=&e_3-c_2+a_2
\end{align*}
One can then check that $\mathsf{V}(I_{\mathsf{init}})\subset \mathsf{V}(I)$.
Now we can call $a_0$, $a_2$ and $c_2$ the essential variables and we can define 
\[
  I_{\mathsf{bric}}=\langle h_0,h_1\rangle
  \subset \mathbb{Q}[a_0,a_2,c_2]
\]
which is an elimination ideal of $I$. 
Then we can call $\mathsf{V}(I_{\mathsf{bric}})\subset \mathbb{R}^3$ Bricard's variety; it  has two real components, see Figure \ref{bric-kuva}, and both curves represent the same physical situation. One can also  check using Theorem \ref{jac-crit} that $\mathsf{V}(I_{\mathsf{bric}}) $ is regular.

\begin{figure}
\begin{center}
\includegraphics[width=120mm]{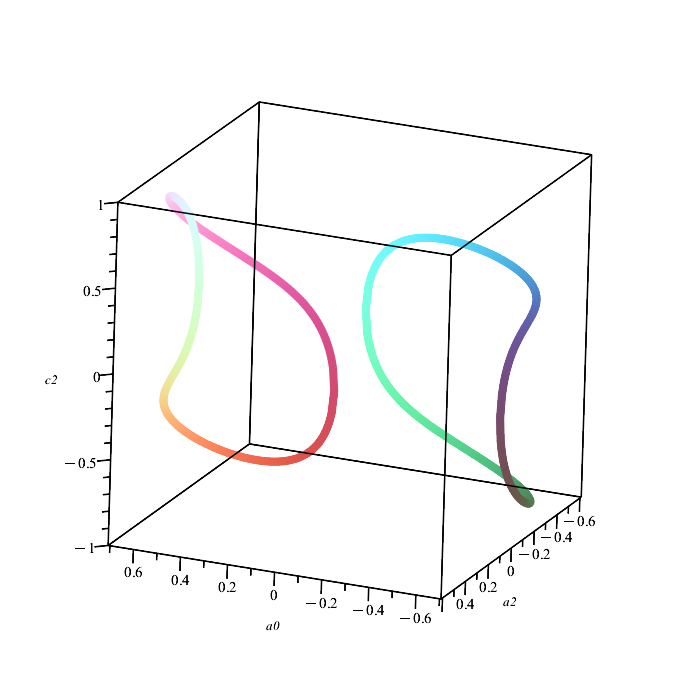}
\end{center}
\caption{The two curves of Bricard's variety both represent the same physical situation.}
\label{bric-kuva}
\end{figure}

 Now we are precisely in the situation as in \eqref{sektio}: the variety $\mathsf{V}(I_{\mathsf{bric}})$ contains the essential information about the whole variety and the above basis provides us with the appropriate map $\mathsf{V}(I_{\mathsf{bric}})\to \mathsf{V}(I)$. Allowing rational functions this map is given even more simply by the following formulas:
 \begin{equation}
\begin{aligned}
  a_1=&a_0&a_3=&-a_2\\
  b_0=&d_1=1/2&
  b_1=&d_0=2a_2(a_2-c_2)\\
  b_2=&d_3=a_2/(2a_0)&
  b_3=&-d_2=2a_0(a_2-c_2)\\
  c_0=&-c_1=a_0-1/(4a_0)&
  c_3=&2a_2-c_2\\
  e_0=&e_1=1/(4a_0)&
  e_2=&e_3=c_2-a_2
\end{aligned}
 \label{bric-para}
 \end{equation}
One can also easily parametrize $\mathsf{V}(I_{\mathsf{bric}})$ for example as follows:
\begin{align*}
    a_0=&\frac{\cos(t)}{\sqrt{2}}&
     a_2=&\frac{\sin(t)}{\sqrt{2}}\\
     c_2= &\frac{\sin(t)}{\sqrt{2}}\pm
     \frac{\sqrt{4\cos^2(t)-1}}{2\sqrt{2}\cos(t)}
\end{align*}
One component of $\mathsf{V}(I_{\mathsf{bric}})$ is obtained when $-\pi/3\le t\le \pi/3$ and the other when $2\pi/3\le t\le 4\pi/3$. Then the parametrization of the whole $\mathsf{V}(I)$ is obtained using the formulas \eqref{bric-para}.

\section{Bennett's mechanism}

Bennett's mechanism has four rigid bodies $\mathcal{B}_\ell$, $0\le j\le 3$, connected by four revolute joints $J_\ell$ so that $\mathcal{B}_\ell$ is between $J_\ell$ and $J_{\ell+1}$ (with $J_4=J_0$).
The location of joints at the initial configuration   are denoted by $p^\ell$. Let us also introduce the vectors $v^\ell=p^{\ell+1}-p^\ell$ with $p^4=p^0$. Then to each joint we associate a vector $\chi^\ell$ which gives the axis of rotation of joint $J_\ell$.

Now the Bennett's mechanism can move only if $v^\ell$ and $\chi^\ell$ satisfy the following conditions:
\begin{enumerate}
\item $|v^0|=|v^2|$ and $|v^1|=|v^3|$.
\item  $\varphi_0=\varphi_2$ and $\varphi_1=\varphi_3$ where $\varphi_\ell$ is the angle between $\chi^\ell$ and $\chi^{\ell+1}$ (with $\chi^4=\chi^0$).
\item  $|v^0|\sin(\varphi_1)=|v^1|\sin(\varphi_0)$.
\end{enumerate}
Let us call these Bennett's conditions. Note that the conditions are invariant with respect to translations, rotations and scaling; hence when discussing Bennett's mechanisms we will always identify two mechanisms, if they can be transformed to each other by translations, rotations and scaling.

We will explicitly consider a certain subclass of Bennett's mechanisms. First we will exclude the planar case: i.e. we require that the points $p^\ell$ are not all on the same plane. The second assumption is more substantial:
\begin{itemize}
\item[($\mathcal{H}$)] we assume that $\chi^\ell$ is orthogonal to $v^\ell$ and $v^{\ell-1}$ (with $v^{-1}=v^3$).
\end{itemize}
Note that Bricard's mechanism satisfies this condition. 
There are two reasons for restricting attention to this case. The first is that with this assumption it is easier to illustrate the constructions described in the previous sections which (we hope) then shows better how to apply this idea in other situations. The second point is that in fact we are able to completely characterize the kinematics of these systems. Our computations will even lead  to an explicit (and simple) parametrization of the configuration space. As far as we know this is a new result.  

Let us call $B_\mathcal{H}$ mechanisms those Bennett's mechanisms which are not planar and  which satisfy in addition the condition $(\mathcal{H})$. Now let us explicitly specify our mechanism. Since the conditions are invariant with respect to translations, rotations and scaling
 we may without loss of generality assume that 
\begin{align*}
  p^0=&(0,0,0)& p^1=&(1,0,0)\\
  p^2=&(1+x_1,x_2,x_3)&p^3=&(y_1,y_2,0)\ .
\end{align*}
The definition of any Bennett's mechanism, with or without condition $(\mathcal{H})$, starts by choosing $x_j$ and $y_j$ such that Bennett's first condition is satisfied.
\begin{lemma} The points $p^\ell$ given above satisfy Bennett's first condition, if 
\begin{align*}
   x_1=& \frac{r(m_1^2-1) }{m_1^2+1} & 
   x_2=&\frac{2rm_1(m_0^2-1)}{(m_0^2 + 1)(m_1^2 + 1)}&
   x_3=&\frac{4rm_0m_1}{(m_0^2 + 1)(m_1^2 + 1)}\\
     y_1=&\frac{r(m_1^2-m_2^2)}{m_1^2+m_2^2}
     & y_2=&\frac{2rm_1m_2}{m_1^2+m_2^2}
\end{align*}
where
\[
r=\frac{(m_0^2+1)(1-m_2^2)}{m_0^2(m_2-1)^2+(m_2+1)^2}\ .
\]
\label{r-lemma}
\end{lemma}
\begin{proof}
The condition $|v^1|=|v^3|=|r|$ is automatically satisfied. Then a simple computation shows that $|v^2|=1$, if $r$ is as above.
\end{proof}

Now the parameters $m=(m_0,m_1,m_2)$ cannot be chosen completely arbitrarily. To exclude the planar case we require that $m_0\ne 0$ and $m_1\ne 0$. Also $r$ cannot be zero so we must have $m_2\ne\pm 1$. Finally we must have $m_2\ne 0$ because otherwise $p^1=p^3$. Let us thus define
\[
  \Omega_0=\big\{ m\in\mathbb{R}^3\,|\,
    m_2\ne \pm 1\ ,\ 
    m_j\ne 0\ \mathrm{for \ all \ } j\big\}\ . 
\]
We call  $\Omega_0$ the set of admissible parameter values. Putting $q=m_0^2(m_2-1)^2+(m_2+1)^2$  we then define a map
\begin{equation}
   \begin{aligned}
     f\,:\  &\Omega_0\to\mathbb{R}^5\quad, \quad
   f(m)=\big(p_1^2(m),p_2^2(m),p_3^2(m),p_1^3(m),p_2^3(m)\big)\\
  p^2=&\frac{2}{q(m_1^2+1)}
  \begin{pmatrix}
    m_0^2(m_2-1)(m_2 - m_1^2)+ (m_2+1)(m_2 + m_1^2)\\
    m_1(m_0^2 - 1)(1-m_2^2 )\\
    2m_0m_1(1-m_2^2)
  \end{pmatrix}\\
  p^3=&\frac{(m_0^2 + 1)(m_2^2 - 1)}{q(m_1^2+m_2^2)}
  \begin{pmatrix}
 m_2^2 - m_1^2\\
  2m_1m_2\\
  0
  \end{pmatrix}
\end{aligned}
\label{f-def}
\end{equation}

\begin{lemma} The map $f$ is not injective. Let us set
 \[
 \Omega=\big\{m\in\mathbb{R}^3\,|\,\  m_0\ge 1\ \mathrm{or}\ m_0<-1 ,\ 
  0<|m_2|<1\ ,\ m_1\ne 0\big\}\ .
\]
Then $f\, :\, \Omega\to\mathbb{R}^5$ is injective.
\label{om-def}
\end{lemma}
\begin{proof}
Let us  define
\begin{align*}
 & g_1\,,\,g_2\,:\,\Omega_0\to\Omega_0\\
  &g_1(m)=\big(m_0,-1/m_1,1/m_2\big)\quad,\quad
  g_2(m)=\big(-1/m_0,-m_1,-m_2\big)\ .
\end{align*}
Note that $g_1$ and $g_2$ generate a group of 4 elements which is  isomorphic to Klein group.
It is now easy to check that
\[
  f=f\circ g_1=f\circ g_2
\]
which shows that $f$ is not injective. It is also straightforward to check that $f$ is injective when restricted to $\Omega$.
\end{proof}

Note that we have $r>0$ when $m\in\Omega$ and it is also easy to check that given some $\hat r>0$ there is some $m_0\ge 1$ and $0<|m_2|<1$ such that $r(m_0,m_2)=\hat r$. Hence, if convenient,  we may  suppose without loss of generality that $m\in\Omega$, and in any case from now on we will always suppose that at least  $m\in\Omega_0$.

 \begin{lemma} Let the points $p^\ell$ be as in Lemma \ref{r-lemma}; then the rotation axes $\chi^\ell$ satisfy the condition $(\mathcal{H})$, if $\chi^0=\mathbf{e}^3$ and 
 \begin{align*}
\chi^1=&\frac{1}{m_0^2+1}\begin{pmatrix}
0\\-2m_0\\m_0^2-1
\end{pmatrix}\\
\chi^2=&\frac{1}{(m_0^2+1)(m_1^2+1)q}
\begin{pmatrix}
4m_0m_1(1-m_2^2 )(m_0^2+1)\\
4m_0\big((m_0^2 - 1)(m_1^2+m_2^2) + m_2(m_0^2 + 1)(m_1^2 + 1)\big)\\
(m_2 - 1)^2(m_1^2 + 1)m_0^4 +
2\big((m_2^2-3 )m_1^2 - 3m_2^2 + 1\big)m_0^2
+(m_2 + 1)^2(m_1^2 + 1)
\end{pmatrix}\\
\chi^3=&\frac{1}{(m_1^2+m_2^2)q}
\begin{pmatrix}
4m_0m_1m_2(1-m_2^2)\\
2m_0(m_1^2 - m_2^2)(1-m_2^2 )\\
(m_0m_2 - m_0 + m_2 + 1)(m_0m_2 - m_0 - m_2 - 1)(m_1^2+m_2^2)
\end{pmatrix}
\end{align*}
Moreover $|\chi^\ell|=1$ for all $\ell$.
\label{khit}
\end{lemma}
\begin{proof}
It is a straightforward computation to check all the statements.
\end{proof}

Note that $|\chi^\ell|=1$ can be achieved without introducing square roots in the formulas which is rather surprising. But even more surprisingly we have
\begin{theorem}
Let $p^\ell$ be as in Lemma \ref{r-lemma} and $\chi^\ell$ as in Lemma \ref{khit}. Then all Bennett's conditions are satisfied. 
\end{theorem}

In other words to each point of $m\in\Omega$ there corresponds a unique $B_{\mathcal{H}}$ mechanism. 

\begin{proof}
By Lemma \ref{r-lemma} we know that the first condition is satisfied. 
Then we compute that
\begin{align*}
  \cos(\varphi_0)=&\llangle \chi^0,\chi^1\rrangle
  =\frac{m_0^2-1}{m_0^2+1}
  =\llangle \chi^2,\chi^3\rrangle=
  \cos(\varphi_2)\\
  \cos(\varphi_1)=&\llangle \chi^1,\chi^2\rrangle
  =\frac{(m_0m_2 - m_0 + m_2 + 1)(m_0m_2 - m_0 - m_2 - 1)}{q}
  =\llangle \chi^3,\chi^0\rrangle=
  \cos(\varphi_3)
\end{align*}
Since we can suppose that $0\le \varphi_j\le \pi$, the cosines  define the angles $\varphi_j$ uniquely and hence Bennett's second condition is satisfied. Then we compute that
\[
\sin(\varphi_1)=
  \frac{2|m_0(m_2^2-1)|}{q}=|r|\sin(\varphi_0)
\]
which is Bennett's third condition.
\end{proof}

Note that the angles $\varphi_j$ do not depend on $m_1$. 

Finally we have to compute the basis for the subspaces which are orthogonal to $\chi^\ell$; we denote this basis by $\{\eta^\ell,\xi^\ell\}$. 

\begin{lemma} For $\ell=0$ we take $\eta^0=\mathbf{e}^1$ and $\xi^0=\mathbf{e}^2$ and for other $\ell$ we set
\begin{align*}
 \eta^\ell=&\Big(\frac{\chi^\ell_3(\chi_1^\ell)^2-(\chi_2^\ell)^2}{1-(\chi_3^\ell)^2}\ ,\ 
 -\frac{\chi_1^\ell\chi_2^\ell}{1+\chi_3^\ell}\ ,\ 
 -\chi_1^\ell\Big)\\
 \xi^\ell=&\Big(-\frac{\chi_1^\ell\chi_2^\ell}{1+\chi_3^\ell}\ ,\ 
 \chi_3^\ell+\frac{(\chi_1^\ell)^2}{1+\chi_3^\ell}\ ,\ 
 -\chi_2^\ell\Big)\ .
\end{align*} 
The basis $\{\eta^\ell,\xi^\ell\}$ is orthonormal and hence the basis  $\{\eta^\ell,\xi^\ell,\chi^\ell\}$ is also orthonormal.
\label{eta-ksi-lem}
\end{lemma}

Since the expressions are quite big in terms of $m$ we do not write them explicitly. 

\begin{proof}
It is again just a big computation to check the correctness of the statements. Note that the formulas are always well defined because  $\chi^\ell$ and $\mathbf{e}^3$ are linearly independent for $\ell\ge 1$. 
\end{proof}

Now the preliminary constructions have been done and we are ready to analyze the kinematics of $B_{\mathcal{H}}$ mechanisms.

\section{Kinematics of $B_{\mathcal{H}}$ mechanisms}

Thus far we have used the global coordinates of $\mathbb{R}^3$ with the standard frame $\{\mathbf{e}^1,\mathbf{e}^2,\mathbf{e}^3\}$. Now we assume that the body $\mathcal{B}_0$ is fixed and so its local frame is the same as the global frame. For other bodies we introduce local frames $\{\mathbf{e}^{1,\ell},\mathbf{e}^{2,\ell},\mathbf{e}^{3,\ell}\}$. Similarly the frame $\{\eta^k,\xi^k,\chi^k\}$ is denoted  $\{\eta^{k,\ell},\xi^{k,\ell},\chi^{k,\ell}\}$ in local coordinates of body $\mathcal{B}_\ell$.

Let us then denote the Euler parameters corresponding to $\mathcal{B}_1$, $\mathcal{B}_2$ and $\mathcal{B}_3$ by $a$, $b$ and $c$. Now we have to compute the orientations, i.e. Euler parameters in the initial configuration. Let us  choose $\mathbf{e}^{1,\ell}$ to be parallel to $v^\ell$. Thus we have to solve the equations
\begin{equation}
 v^1=r\, R(a)\mathbf{e}^{1,1}\ ,\ 
    v^2= R(b)\mathbf{e}^{1,2}\ ,\ 
     v^3=r\, R(c)\mathbf{e}^{1,3}\ .
\label{init-yht}
\end{equation}
where $r$ is as in Lemma \ref{r-lemma}.  Let $\mathbb{K}=\mathbb{Q}(m_0,m_1,m_2)$ be the field of rational functions in variables $m_0$, $m_1$ and $m_2$. Then we have
\begin{lemma} One can choose following solutions to equations \eqref{init-yht}:
\begin{align*}
  a_{\mathsf{init}}=&
  \frac{1}{m_1^2+1}\,\Big(
    m_1,   m_1^2,
    \frac{q_5}{m_0^2 + 1},
     \frac{m_0^2+2m_0m_1-1}{m_0^2 + 1}
     \Big)   \\
      b_{\mathsf{init}}=&\frac{1}{(m_1^2+1)(m_1^2 + m_2^2)q}\Big(
    m_1(m_2^2-1)q_0, 
     m_1(m_2^2-1)q_1,\\
& (1-m_2) (m_1^2-m_2)q_1+2q_2-\frac{8(m_1^2+m_2^2)(m_1^2m_0+m_2m_1+m_2m_0-m_1)}{m_0^2+1}   , \\&  (1-m_2)(m_1^2-m_2)q_0+2 q_3+ \frac{8(m_1^2+m_2^2)(m_2m_1m_0-m_1^2-m_1m_0-m_2)}{m_0^2 + 1}
     \Big)   \\
   c_{\mathsf{init}}=&\frac{1}{m_1^2+m_2^2}\Big(
    -m_1m_2,   -m_2^2,
    m_1m_2,  m_1^2
     \Big)     
\end{align*}
where
\begin{align*}
  q_0=&
  q_4(m_1^2 + m_2) - 2m_0(m_2 - 1)m_1\\
  q_1=&
  q_4m_1(m_2 - 1) +2m_0(m_1^2 + m_2)\\
  q_2=& 
  (1-m_2)(m_1^2+m_2^2)
 \big(m_1^2m_0+m_2m_1+m_2m_0-3m_1\big) \\
  q_3=& (m_1^2+m_2^2)\big(m_2m_1(m_2m_0-m_1-2m_0)-m_2^2+3m_1^2+m_1m_0+3m_2\big)  \\
  q_4=& 
m_0^2(m_2-1)+m_2+1\\
q_5=&m_1(m_0^2-1)-2m_0
\end{align*}
\label{alku-lemma}
\end{lemma}

\begin{proof}
Let us consider the first case and let
\begin{align*}
   h=&v^1-r\, R(a)\mathbf{e}^{1,1}\\
   I=&\langle h_1,h_2,h_3,|a|^2-1\rangle \subset \mathbb{K}[a_3,a_2,a_1,a_0]
\end{align*}
Computing the Gr\"obner basis with lex order gives
\[
  G=\{
  (m_1^2 + 1)(a_0^2 + a_1^2) - m_1^2, 2a_0m_0 + a_1(1-m_0^2 ) + m_1a_2(m_0^2 + 1), a_0(1-m_0^2 ) - 2a_1m_0 + m_1a_3(m_0^2 + 1)\}
\]
This gives
\begin{align*}
   a_2=& \frac{a_1m_0^2 - 2a_0m_0 - a_1}{m_1(m_0^2 + 1)} &
   a_3=& \frac{a_0m_0^2 + 2a_1m_0 - a_0}{m_1(m_0^2 + 1)} \\
   a_0^2+a_1^2=&\frac{m_1^2}{m_1^2+1}
\end{align*}
and one can readily check that $ a_{\mathsf{init}}$ given above is one solution. 
Computing similarly with  $ v^3=r\, R(c)\mathbf{e}^{1,3}$ gives
\[
  c_0^2+c_1^2=\frac{m_2^2}{m_1^2+m_2^2}\quad,\quad
     c_2=-\frac{m_1}{m_2}\,c_1
     \quad,\quad
     c_3=-\frac{m_1}{m_2}\,c_0
\]
and $c_{\mathsf{init}}$ provides one solution. Finally the equations $ v^2= R(b)\mathbf{e}^{1,2}$ are more complicated, but the structure of the problem is precisely the same: after computing the Gröbner basis we get the equations
\[
  b_0^2+b_1^2=\hat\gamma_0
  \quad,\quad
  b_2=\hat\gamma_1b_0+\hat\gamma_2 b_1
    \quad,\quad
  b_3=\hat\gamma_3b_0+\hat\gamma_4 b_1
\]
where $\hat\gamma_j\in\mathbb{K}$.  Now let us put $\hat\gamma_0=\kappa_1/\kappa_0$ where $\kappa_j$ are polynomials and $b_j=\hat b_j/\kappa_0$. Then
\begin{equation}
    \hat b_0^2+\hat b_1^2=\kappa_0\kappa_1=
     m_1^2(m_2^2-1)^2
 (m_0^2+1)(m_1^2+1)(m_1^2+m_2^2)
 \big(m_0^2(m_2-1)^2+(m_2+1)^2\big)\ .
\label{outo}
\end{equation}
But since $\kappa_0\kappa_1$ is of this particular form we can use repeatedly the identity
\[
   (k_0^2+k_1^2)(k_2^2+k_3^2)=
   (k_0k_2+k_1k_3)^2+(k_0k_3-k_1k_2)^2
\]
which then gives the formulas for $b_0$ and $b_1$ given above. Then $b_2$ and $b_3$ can be computed using $ b_2=\hat\gamma_1b_0+\hat\gamma_2 b_1
 $ and $
  b_3=\hat\gamma_3b_0+\hat\gamma_4 b_1$.
\end{proof}

\begin{remark}
In \eqref{outo} we succeeded in expressing the positive polynomial on the right hand side as a sum of squares of just two polynomials. This is possible because the right hand side happened to be of very particular form. However, there seems to be no reason why the right hand side should be of this form.
\hfill \ding{71}
\end{remark}

We can now define  the initial ideal 
\[
  I_{\mathsf{init}}=\langle a-a_{\mathsf{init}},
  b-b_{\mathsf{init}},c-c_{\mathsf{init}}\rangle
  \subset \mathbb{K}[a,b,c]
\]
whose variety $\mathsf{V}(I_{\mathsf{init}})$ gives the point corresponding to the initial configuration.

Then we need to compute the appropriate vectors $\eta^{k,\ell}$, $\xi^{k,\ell}$ and $\chi^{k,\ell}$ which are given by the following formulas:
\begin{align*}
 \chi^{1,1}=&R\big(a_{\mathsf{init}}\big)^T\chi^1&
   \chi^{2,2}=&R\big(b_{\mathsf{init}}\big)^T\chi^2&
   \chi^{3,3}=&R\big(c_{\mathsf{init}}\big)^T\chi^3\\
    \eta^{2,1}=&R\big(a_{\mathsf{init}}\big)^T\eta^2&
   \eta^{3,2}=&R\big(b_{\mathsf{init}}\big)^T\eta^3&
   \eta^{0,3}=&R\big(c_{\mathsf{init}}\big)^T\eta^0\\
       \xi^{2,1}=&R\big(a_{\mathsf{init}}\big)^T\xi^2&
   \xi^{3,2}=&R\big(b_{\mathsf{init}}\big)^T\xi^3&
   \xi^{0,3}=&R\big(c_{\mathsf{init}}\big)^T\xi^0
\end{align*}
Now we are in the position to define the polynomials which define our computational problem. First we have the normalization polynomials:
\[
  f_0=|a|^2-1\quad,\quad f_1=|b|^2-1\quad\mathrm{and}
  \quad f_2=|c|^2-1\ .
\]
Then we have loop polynomials:
\[
 \begin{pmatrix}
 f_3\\f_4\\f_5
 \end{pmatrix}=
 \Big(I+r\, R(a)+
     R(b)+
    r\, R(c)\Big)\mathbf{e}^1
 \]
Then there are polynomials for each of the joints:
\begin{align*}
   f_6=&\llangle \eta^1,R(a)\chi^{1,1}\rrangle  &
     f_7=& \llangle \xi^1,R(a)\chi^{1,1}\rrangle    \\
  f_8=& \llangle R(a)\eta^{2,1},R(b)\chi^{2,2}\rrangle   &
    f_9=&  \llangle R(a)\xi^{2,1},R(b)\chi^{2,2}\rrangle \\
  f_{10}=& \llangle R(b)\eta^{3,2},R(c)\chi^{3,3}\rrangle  & 
   f_{11}=& \llangle R(b)\eta^{3,2},R(c)\chi^{3,3}\rrangle  \\
     f_{12}=&  \llangle R(c) \eta^{0,3},\chi^0\rrangle   &
     f_{13}=&   \llangle R(c) \xi^{0,3},\chi^0\rrangle 
 \end{align*}  
So we have 14 equations in 12 variables so it is easy to believe that Bennett's mechanism can move only in special configurations. Let us then introduce the following ideals:
\begin{align*}
  I_a=&\langle f_6,f_7,f_0\rangle &
  I_c=&\langle f_{12},f_{13},f_2\rangle\\
   I_{ab}=&\langle f_8,f_9,f_0,f_1\rangle &
  I_{bc}=&\langle f_{10},f_{11},f_1,f_2\rangle\\
  I_{abc}=&\langle f_3,f_4,f_5\rangle&
  I_0=&I_a+I_c+I_{ab}+I_{bc }+I_{abc}
\end{align*}
Our goal is to analyze $I_0$ and the corresponding variety $\mathsf{V}(I_0)$. Now we know at the outset that $I_0$ is not prime, and so the corresponding variety is reducible. As in the Bricard's case we use the formula \eqref{dekomp} for each of the joints, in other words for ideals $I_a$, $I_c$, $I_{ab}$ and  $I_{bc}$, to compute the prime decomposition and to find the correct prime component.

Let us first decompose  $I_a$ and $I_c$.
These ideals are so simple that one can directly compute their prime decompositions which gives
\begin{equation}
\begin{aligned}
  I_a=&J_a\cap \hat J_a\subset \mathbb{K}[a_3,a_1,a_2,a_0]\quad,\quad 
  a_{\mathsf{init}}\in \mathsf{V}(J_a)
  \quad \mathrm{where}\\
  J_a=&\langle a_1-m_1a_0,
  q_5a_3-(m_0^2+2m_0m_1-1)a_2,\\
 &   (m_1^2+1)q_5^2a_0^2+
   (m_1^2+1)(m_0^2+1)^2a_2^2-q_5^2
    \rangle
\end{aligned}
\label{dec-a}
\end{equation}
Note that here we require that $q_5\ne 0$. We treat the case $q_5=0$ later. 
Similarly we have 
\begin{equation}
\begin{aligned}
  I_c=&J_c\cap \hat J_c
  \subset \mathbb{K}[c_3,c_1,c_2,c_0]
  \quad,\quad 
  c_{\mathsf{init}}\in  \mathsf{V}(J_c)
  \quad \mathrm{where}\\
  J_c=&\langle m_1c_1-m_2c_0,
  m_2c_3-m_1c_2,
   m_2^2(m_1^2+m_2^2)c_0^2+
    m_1^2(m_1^2+m_2^2)c_2^2-m_1^2m_2^2
   \rangle
\end{aligned}
\label{dec-c}
\end{equation}
 We have now obtained 
\[
  I_0\subset I_1=J_a+J_c+I_{ab}+I_{bc }+I_{abc}
\]
such that $\mathsf{V}(I_{\mathsf{init}})\subset \mathsf{V}(I_1)$ 
and we proceed with $I_1$. 

Now we use formulas \eqref{muunnos} together with the special decomposition \eqref{hajo} to compute the decompositions of $I_{ab}$ and $I_{bc}$. Here the expressions are simply so big that it is impossible to include them here. However, the computations are very fast, since here one does not need any Gr\"obner basis computations. 

After identifying the correct prime component, denoted by $J_{ab}$ and $J_{bc}$, we have obtained the ideal 
\[
  I_2=J_a+J_c+J_{ab}+J_{bc}+I_{abc}\subset \mathbb{K}[a,b,c]
\]
Now we should compute the Gr\"obner basis of $I_2$. However, it is first best to eliminate as many variables as possible before this. Using \eqref{dec-a} and \eqref{dec-c} we have 
\begin{equation}
  a_1=m_1a_0\ , \ a_3=\frac{m_0^2+2m_0m_1-1}{q_5}\,a_2\ ,\ 
  c_1=\frac{m_2}{m_1}\,c_0\ ,\ 
  c_3=\frac{m_1}{m_2}\,c_2\ .
\label{noness}
\end{equation}
Substituting these expressions to $I_2$ we obtain
\[
      I_3\subset \mathbb{K}[b_3,b_2,b_1,b_0,c_2,c_0,a_2,a_0]
\]
Now we compute the Gr\"obner basis of this using the degree reversed lex order. The computation took about one minute with the standard laptop. The basis has 23 elements and the dimension of $\mathsf{V}(I_3)$  is one as expected.

The elements of the basis are quite simple as polynomials in $a$, $b$ and $c$; however, their coefficients which are elements of $\mathbb{K}$ are quite large so that we do not write down the whole  basis explicitly. Anyway by looking at the basis more closely we can analyze the situation further. In the following $\gamma_j\in\mathbb{K}$ are coefficients whose expressions are so big that they are not written explicitly. 
\begin{equation}
\begin{aligned}
  I_3=&\langle g_0,\dots,g_{22}\rangle
  \subset \mathbb{K}[b_3,b_2,b_1,b_0,c_2,c_0,a_2,a_0]\\
  g_0=& a_0^2+
   \frac{(m_0^2+1)^2}{q_5^2}\,a_2^2-\frac{1}{m_1^2+1}\\
   g_1=& m_2^2c_0^2+
    m_1^2c_2^2-\frac{m_1^2m_2^2}{m_1^2+m_2^2}\\
    g_2=& b_1+\gamma_0 b_0\\
    g_3=&m_1(m_0^2+1)c_2a_2+q_5c_0a_0\\
    g_4=&
    (m_2^2-1)(m_1^2+1)(m_1^2+m_2^2)c_0^2a_0^2-m_2^2(m_1^2+m_2^2)c_0^2-
    m_1^2m_2^2(m_1^2+1)a_0^2+m_1^2m_2^2\\
    g_5=&b_0^2+\gamma_1c_0^2+\gamma_2a_0^2+\gamma_3\\
    &\vdots
\end{aligned}
\label{id-i3}
\end{equation}
Of course we expected that $g_0$  ( from  \eqref{dec-a}) and $g_1$ (from \eqref{dec-c})) are in the basis. 

\begin{remark} Let us add here a technical detail. In the above basis for $I_3$ the  variable  $b_1$ could explicitly be expressed in terms of  $b_0$. Now in the monomial order that was used $b_1$ was bigger than $b_0$ and since we computed the reduced Gröbner basis (see \cite{cox-li-os})  the variable $b_1$ was eliminated from other elements of the basis. \hfill 
\ding{71}
\end{remark}

Now $g_0$, $g_1$ and $g_3$ generate a one dimensional ideal in $\mathbb{K}[c_2,c_0,a_2,a_0]$ which we hope would be the appropriate elimination ideal as in \eqref{sektio}.  So let us set
\[
  I_{\mathsf{benn}}=\langle g_0,g_1,g_3\rangle
  \subset \mathbb{K}[c_2,c_0,a_2,a_0]
\]
and let us call $\mathsf{V}(I_{\mathsf{benn}})$ Bennett's variety. To construct the map $h$ as in \eqref{sektio} we should express $b_j$ in terms of $(a_0,a_2,c_0,c_2)$.

From $g_2$ we already know that $b_1$ can be expressed in terms of $b_0$. Then looking more closely at $g_4$ and $g_5$ one notices that
\begin{align*}
      g_4=&
    (m_2^2-1)(m_1^2+1)(m_1^2+m_2^2)c_0^2a_0^2-m_2^2\Big((m_1^2+m_2^2)c_0^2+
    m_1^2(m_1^2+1)a_0^2-m_1^2\Big)\\
    g_5=&b_0^2-\gamma_4\Big((m_1^2+m_2^2)c_0^2+
    m_1^2(m_1^2+1)a_0^2-m_1^2\Big)
\end{align*}
which immediately gives
\[
   b_0^2=\frac{ (m_2^2-1)(m_1^2+1)(m_1^2+m_2^2)\gamma_4}{m_2^2}\,c_0^2a_0^2\ . 
\]
But then we note that $(m_2^2-1)(m_1^2+1)(m_1^2+m_2^2)\gamma_4$ is a perfect square which combined with $g_2$ yields
\begin{equation}
\begin{aligned}
   b_0=&\pm t_0a_0c_0& b_1=&\pm t_1a_0c_0&
   \mathrm{where}\\
       t_0=& \frac{(m_2^2-1)q_0}{m_1m_2q}&  t_1=&\frac{(m_2^2-1)q_1}{m_1m_2q} \ .
\end{aligned}
\label{t0t1}
\end{equation}
\begin{remark} It is quite unexpected that that there is $\gamma_4$ such that $g_4$ and $g_5$ are related in this way. Even more remarkable is that $(m_2^2-1)(m_1^2+1)(m_1^2+m_2^2)\gamma_4$  happens to be a perfect square. 
\hfill  \ding{71}
\end{remark}

From this it follows that the ideal $I_3$ is not prime and we can write
\[
  I_3=I_{3,0}\cap I_{3,1}=
    \Big(I_3+\langle b_0-t_0a_0c_0, b_1-t_1a_0c_0\rangle\Big)\cap
      \Big(I_3+\langle b_0+t_0a_0c_0, b_1+t_1a_0c_0\rangle\Big)\ .
\]
Now $I_{3,0}$ and $ I_{3,1}$ represent the same physical situation because it turns out that in the final result it is simply a question of the choice of the sign of $b$. Since $R(b)=R(-b)$ it does not matter which sign is chosen. Let us choose for example  $I_{3,0}$; 
then substituting $b_0=t_0a_0c_0$ and $b_1=t_1a_0c_0$ to $I_3$ gives us an ideal $I_4\subset \mathbb{K}[b_3,b_2,c_2,c_0,a_2,a_0]$ and computing the Gröbner basis gives 
\begin{align*}
I_4=&\langle g_0,g_1,g_3,\hat g_0,\dots,\hat g_{16}\rangle \subset \mathbb{K}[b_3,b_2,c_2,c_0,a_2,a_0]\\
\hat g_0=& a_0 \tilde g_0\quad,\quad
\hat g_1= c_0 \tilde g_1\quad,\quad
\hat g_2= c_0 \tilde g_2\\
         &\vdots
\end{align*}
Hence the ideal $I_4$ is still  not prime and we have
\[
I_4=I_5\cap I_{5,0}\cap I_{5,1}=
  \Big(I_4+\langle \tilde g_0,
  \tilde g_1,\tilde g_2\rangle\Big)\cap 
\Big(I_4+\langle a_0\rangle\Big)\cap
   \Big(I_4+\langle c_0\rangle\Big) \ .
\]
Computing the Gröbner bases for $ I_{5,k}$ reveals that they are zero dimensional; the varieties $\mathsf{V}(I_{5,k})$ consist of 8 points each. However, these 8 points correspond to same physical situation: they correspond to 8 possible choices of signs in $\pm a$, $\pm b$ and $\pm c$. Now it turns out that in each case four of these points are contained in $\mathsf{V}(I_5)$ and hence from the point of view of actual mechanism $\mathsf{V}(I_{5,k})$ correspond to some particular points in $\mathsf{V}(I_5)$. Note that these points are in no way special from the point of view of the mechanism itself and so we can simply dismiss $ I_{5,k}$ and turn our attention to $I_5$. Then computing the Gröbner basis of $I_5$ gives
\begin{align*}
  I_5=&\langle g_0,g_1,g_3,\bar g_0,\dots,\bar g_{11}\rangle \subset \mathbb{K}[b_3,b_2,c_2,c_0,a_2,a_0]\\
 \bar g_0=&b_3+\gamma_5 b_2+\gamma_6 a_0c_2\\
   \bar g_1=&b_3+\gamma_7 b_2+\gamma_8 a_2c_0\\
   \vdots
\end{align*}
But now we can easily get expressions for $b_2$ and $b_3$ using $\bar g_0$ and $\bar g_1$;  we obtain
\begin{equation}
\begin{aligned}
  b_2=& t_2a_0c_2+ t_3a_2c_0& 
  b_3=& t_4a_0c_2+ t_5a_2c_0\\
    t_2=& \frac{m_1(m_2 - 1)(q_4^2 - 4m_0^2) + 4m_0q_4(m_1^2 + m_2)}{m_2(m_0^2+1)q} \\
    t_3=&\frac{m_1(m_2-1)(m_0^2-1)-2m_0(m_2+m_1^2)}{m_1q_5}\\
    t_4=&  \frac{(m_1^2 + m_2)(q_4^2 - 4m_0^2) - 4m_0m_1(m_2 - 1)q_4}{m_2(m_0^2+1)q}\\
    t_5=& \frac{(m_2+m_1^2)(m_0^2-1)+2m_0m_1(m_2-1)}{m_1q_5}
\end{aligned}
\label{teet}
\end{equation}
Let us summarize the results. It has turned out that $a_0$, $a_2$, $c_0$ and $c_2$ are (or can be chosen as) essential variables and all other variables can be expressed explicitly in terms of them. In terms of ideals we can set 
\begin{align*}
  I_{\mathsf{final}}=&I_5+\langle \bar f_0,\dots,\bar  f_7\rangle
  \subset \mathbb{K}[a,b,c]\\
  \bar f_0=&b_0-t_0a_0c_0&
  \bar f_1=&b_1-t_1a_0c_0\\
  \bar f_2=& b_2- t_2a_0c_2- t_3a_2c_0& 
  \bar f_3=&b_3- t_4a_0c_2- t_5a_2c_0\\
 \bar f_4=&  a_1-m_1a_0&
 \bar f_5=& a_3-\frac{m_0^2+2m_0m_1-1}{q_5}\,a_2\\ 
 \bar f_6=& c_1-\frac{m_2}{m_1}\,c_0& 
\bar f_7=&  c_3-\frac{m_1}{m_2}\,c_2
\end{align*}
Using now the generators $\bar f_j$ we have the map
\begin{equation}
  h\ : \ \mathsf{V}(I_{\mathsf{benn}})\to 
  \mathsf{V}(I_{\mathsf{final}})
\label{benn-sektio}
\end{equation}
as in \eqref{sektio}. 

The Bennett's variety has actually two real components, but they both represent the same physical situation so it does not matter which one is chosen. One also easily checks that the  Bennett's variety is regular, using Theorem \ref{jac-crit}.  In Figure \ref{benn-kuva} we have projected $\mathsf{V}(I_{\mathsf{benn}})$ to $(a_0,a_2,c_0)$ space with certain parameter values $m$ and it is seen that the projected curves intersect, but of course the original curves in four dimensional space do not intersect.

We can now easily obtain an explicit parametrization of $\mathsf{V}(I_{\mathsf{benn}})$. The polynomial $g_0$ in \eqref{id-i3} implies that $a_0^2\le 1/(m_1^2+1)$ and this suggests that we could set $a_0=\cos(t)/\sqrt{m_1^2+1}$. One component of $\mathsf{V}(I_{\mathsf{benn}})$ is then given by
\begin{align*}
   a_0=&\frac{\cos(t)}{\sqrt{m_1^2+1}}&
   a_2=&-\frac{q_5\sin(t)}{(m_0^2+1)\sqrt{m_1^2+1}}\\
   c_0=&\frac{m_1m_2\sin(t)}{\sqrt{q_6}}&
   c_2=&\frac{m_2\cos(t)}{\sqrt{q_6}}\\
   q_6=&(m_1^2+m_2^2)(m_2^2\sin(t)^2+\cos(t)^2)
\end{align*}
This with the map $h$ in \eqref{benn-sektio} gives the parametrization of one component of $\mathsf{V}(I_{\mathsf{final}})$ and it is   straightforward to check that the initial point is obtained by choosing $\cos(t)=m_1/\sqrt{m_1^2+1}$ and $\sin(t)=-1/\sqrt{m_1^2+1}$. 
The other component is obtained by changing the sign of $c$ in this parametrization.

Now we remarked earlier that the formulas are a bit  different if parameters are chosen such that $q_5=0$. Actually there is really nothing geometrically special about this case: it simply means that algebraically we should choose $a_3$ as an essential variable instead of $a_2$ because $a_2=0$ in this case. Otherwise one can compute precisely as above except that formulas are somewhat simpler because using $q_5=0$ one can eliminate the parameter $m_1$ by $m_1=2m_0/(m_0^2-1)$.
The final result is
\begin{align*}
  &\frac{(m_0^2+1)^2}{(m_0^2-1)^2}\, a_0^2+a_3^2-1=0\\
  &  m_2^2c_0^2+\frac{4m_0^2}{(m_0^2-1)^2}c_2^2-\frac{4m_0^2m_2^2}{(m_2^2-1)^2m_2^2+4m_0^2}=0\\
  &2m_0 c_2a_3 + (m_0^2+1)c_0a_0=0\\
    a_1= &\frac{2m_0}{m_0^2-1} a_0\ , \ a_2=0\ ,&\ 
  c_1=&\frac{m_2(m_0^2-1)}{2m_0}\,c_0\ ,\ 
  c_3=\frac{2m_0}{m_2(m_0^2-1)}\,c_2\\
      b_0=&\hat t_0 a_0c_0&b_1=&\hat t_1a_0c_0\\
    b_2=&\hat t_2a_0c_2+\frac{m_2(m_0^2-1)}{2m_0}\,a_3c_0  & b_3=&\hat t_4a_0c_2-a_3c_0
\end{align*}
where $\hat t_j$ are obtained by simply substituting $m_1=2m_0/(m_0^2-1)$ to the formulas of $t_j$.
Starting now with $a_0=(m_0^2-1)\cos(t)/(m_0^2+1)$ explicit parametrizations can be  obtained in the same way as before.

\begin{figure}
\begin{center}
\includegraphics[width=120mm]{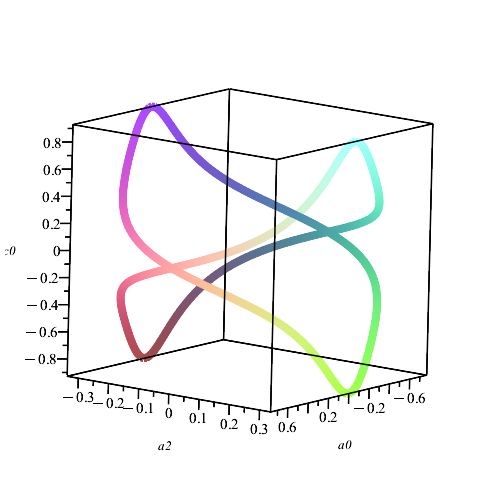}
\end{center}
\caption{The projection of Bennett's variety to $(a_0,a_2,c_0)$ space with $m=(7/4,5/6,1/3)$. }
\label{benn-kuva}
\end{figure}

\section{Conclusion}

Let us interpret our results from the point of view of mobility (or the number of degrees of freedom) of the mechanisms, see \cite{gogu} for a thorough overview of this topic. When the problem is formulated in terms of ideals and varieties the mobility is simply the dimension of the relevant variety. As we have seen in general ideals are not prime, and in the prime decomposition which gives the decomposition of the variety to irreducible varieties there may very well be varieties of different dimensions. So in the analysis of a given mechanisms one needs to determine which is the relevant prime ideal/irreducible variety which one is interested in; then the mobility is the dimension of this component. The (local) mobility can also be computed numerically \cite{wampler2}.

Now as discussed in \cite{gogu} there have been attempts to compute the mobility without actually analysing the constraint equations, for example the well-known Gr\"ubler-Kutzbach-Chebyshev formula. All such formulas fail for some mechanisms, in particular they fail for Bricard's and Bennett's mechanisms. Another way of failure of counting arguments is to try to count the number of constraint equations: for Bricard's mechanism we have (initially) 20 equations with 20 unknowns and for Bennett's mechanism we have (initially) 14 equations with 12 unknowns. If the equations are independent then one expects that neither of these mechanisms can move. However, this reasoning which is valid for linear equations is not at all valid for nonlinear equations. Hence one has called such mechanisms whose mobility is not what counting arguments suggest as overconstrained or even paradoxical. 

But counting the number of equations/constraints/generators is not really important for the kinematic analysis when the problem is formulated in terms of polynomial ideals: the important thing is to compute the Gr\"obner basis, and the number of generators of the basis is irrelevant. Indeed when one changes the monomial order the number of generators  often changes. Hence the concept of overconstrained is not intrinsically related to the mechanism, it only describes certain representations of mechanisms. Above we saw that $\mathsf{V}(I_{\mathsf{benn}})$ is actually a complete intersection, so together with the map \eqref{benn-sektio} we have a perfectly standard way to represent the Bennett's mechanism. Similarly $\mathsf{V}(I_{\mathsf{bric}})$ is a complete intersection and with formulas \eqref{bric-para} we have a standard representation also for Bricard's mechanism. Intuitively one might say that the whole difficulty is due to the fact that the initial constraints give us an ideal which is not prime, so that the representation allows also spurious components in the irreducible decomposition of the variety. Once we have got rid of these spurious components the problems disappear and we have a nice representation of the physical situation.

In the dynamical simulations the number of equations is important, in an indirect way. Let $\mathcal{I}=\langle f_1,\dots,f_k\rangle$ be a given ideal,  $V=\mathsf{V}(\mathcal{I})$ and let $f=(f_1,\dots,f_k)$ be the corresponding map. If $V$ is a complete intersection so that $\mathsf{codim}(V)=k$ then $df$ is of full rank on $V$, if $V$ is regular and on $V\setminus \Sigma(V)$ otherwise. Now the standard simulation methods for constrained dynamical systems require that $f$, which represents the constraints, is such that $df$ is of full rank. This is why the simulation of Bricard's or Bennett's mechanisms are problematic when one uses the original constraints; however, above we have computed the relevant constraints in both cases such that the corresponding $df$ has a full rank, and moreover the relevant varieties are regular. Hence using our formulation any standard simulation tools can be used to study these systems.

Above we have used Bricard's and Bennett's mechanisms to illustrate our approach, but our method can be applied to any mechanism which has revolute joints. As we have seen the computational effort reduces significantly using this idea which means that much more complicated mechanisms can be analyzed using Gr\"obner basis methods than previously was possible.

\printbibliography

\end{document}